\documentclass[a4paper, leqno]{amsart}

\usepackage{amssymb, amsmath, amsthm}
\usepackage{hyperref}
\usepackage[numbers]{natbib}
\usepackage{xcolor}
\numberwithin{equation}{section}
\newtheorem{thm}{Theorem}
\newtheorem{theo}[thm]{Theorem}

\newtheorem{lemma}[thm]{Lemma}
\newtheorem{proposition}[thm]{Proposition}

\newtheorem{rmk}[thm]{Remark}

\newtheorem*{theorem*}{Theorem}

\newtheorem*{definition}{Definition}
 
\newcommand{\dd}{{d}}
\newcommand{\T}{\mathbb{T}}
\newcommand{\norm}[1]{\left\Vert#1\right\Vert}
\newcommand{\brkt}[1]{\left(#1\right)}

\newcommand{\abs}[1]{\left|#1\right|}

\title[Notes on $H^{\log}$]{Notes on $H^{\log} $: structural properties, dyadic variants, and bilinear $H^1$-$BMO$ mappings}

\author[Bakas]{Odysseas Bakas}
\address{Centre for Mathematical Sciences, Lund University, 221 00 Lund, Sweden}
\email{odysseas.bakas@math.lu.se}
\author[Pott]{Sandra Pott}
\address{Centre for Mathematical Sciences, Lund University, 221 00 Lund, Sweden}
\email{sandra.pott@math.lu.se}
\author[Rodr\'iguez-L\'opez]{Salvador Rodr\'iguez-L\'opez}
\address{Department of Mathematics, Stockholm University, 106 91 Stockholm, Sweden}
\email{s.rodriguez-lopez@math.su.se}
\author[Sola]{Alan Sola}
\address{Department of Mathematics, Stockholm University, 106 91 Stockholm, Sweden}
\email{sola@math.su.se}
\date{\today}
\thanks{The first author was partially supported by the `Wallenberg Mathematics Program 2018', grant no. KAW 2017.0425, financed by the Knut and Alice Wallenberg Foundation.
The second author was partially supported by VR grant 2015-05552.
The third author was partially supported by the Spanish Government grant MTM2016-75196-P}


\begin{document}

\maketitle

\begin{abstract}This article is devoted to a study of the Hardy space $H^{\log} (\mathbb{R}^d)$ introduced by Bonami, Grellier, and Ky. We present an alternative approach to their result relating the product of a function in the real Hardy space $H^1$ and a function in $BMO$ to distributions that belong to $H^{\log}$ based on dyadic paraproducts.
We also point out analogues of classical results of Hardy-Littlewood, Zygmund, and Stein for $H^{\log}$ and related Musielak-Orlicz spaces.
 \end{abstract}

\section{Introduction}
The Lebesgue spaces $L^p$ with $1\leq p\leq \infty$ are fundamental in mathematical analysis. They are easy to define, they are Banach spaces with many useful properties, and they encode a natural notion of regularity in a measurable function. Nevertheless, there are many instances where $L^p$ spaces, especially with $p=1$, do not capture finer properties of functions or operators acting on functions. In such instances, it may be necessary to consider substitutes for $L^1$, as is the case when studying endpoint bounds for operators on $L^p$ as $p\to 1^+$. 

For instance, the ubiquitous Hardy-Littlewood maximal function exhibits precisely this kind of behaviour near $p=1$. The maximal function is defined for a locally integrable function $f\colon \mathbb{R}^d\to \mathbb{C}$ by setting
$$ M (f) (x) := \sup_{r>0} \frac{1}{|B (x,r)|} \int_{B(x,r)} |f(y)| dy \quad  x \in \mathbb{R}^d,$$
where $B(x,r)$ denotes the open ball in $\mathbb{R}^d$ centered at $x$ with radius $r>0$, and $|A|$ denotes the Lebesgue measure of $A \subseteq \mathbb{R}^d$. It is a basic fact that the mapping $f\mapsto M(f)$ is bounded on $L^p(\mathbb{R}^d)$ for $1<p\leq \infty$. The maximal operator is also bounded from $L^1(\mathbb{R}^d)$ to weak-$L^1$, but does not map $L^1(\mathbb{R}^d)$ to itself (see, for instance, \cite{Big_Stein} for an in-depth discussion).

However, $M(f)$ is locally integrable provided $f$ is compactly supported and satisfies the $L\log L$ condition
$$ \int_{\mathbb{R}^d}|f(x)|\log^+|f(x)|dx<\infty,$$
where, as usual, $\log^+|x|=\max\{\log|x|, 0\}$.
In a 1969 paper, E. M. Stein \cite{Stein_LlogL} proved that this $L\log L$ condition is both sufficient and necessary for integrability of the Hardy-Littlewood maximal function, in the following sense: if $f$ is supported in some finite ball $B=B(r)$ of radius $0<r<\infty$, then 
$$ \int_{B}M(f) dx<\infty \quad \textrm{if, and only if,}\quad \int_{B}|f(x)|\log^+|f(x)|dx<\infty. $$
Thus, $L\log L$ is a natural substitute for $L^1$ for the purposes of studying the boundedness of the Hardy-Littlewood maximal function in the scale of $L^p$ spaces.  

Another classical result that involves the space $L\log L$ is due to Zygmund, and asserts that the periodic Hilbert transform $H$ maps $L\log L(\mathbb{T})$ to\footnote{Notice that the aforementioned endpoint bounds for $M$ and $H$ can be regarded as special cases of the fact that if $T$ is any sublinear operator that is bounded on $L^{p_0}$ for some $p_0 > 1$ and maps $L^1 $ to weak-$L^1$, then $T$ locally maps $L \log L$ to $L^1$.} $L^1(\mathbb{T})$; see e.g. Theorem 2.8 in Chapter VII of \cite{Zygmund_book}. Zygmund's result implies that $L\log L(\mathbb{T})$ is contained in the real Hardy space $H^1(\mathbb{T})$ consisting of integrable functions on the torus whose Hilbert transforms are integrable. Moreover, as shown by Stein in \cite{Stein_LlogL}, Zygmund's theorem has a partial converse, namely if $f\in H^1(\mathbb{T})$ and $f$ is non-negative, then $f$ necessarily belongs to $L\log L(\mathbb{T})$. Therefore, in view of the aforementioned results of Zygmund and Stein, the Hardy  space $H^1(\mathbb{T})$ is, in terms  of  magnitude, associated with the Orlicz space $L\log L(\mathbb{T})$. 

In several problems in harmonic analysis it is natural to consider Hardy-Orlicz spaces, for instance when one studies certain problems related to endpoint mapping properties of operators; see e.g. \cite[Theorem 8]{Zygmund}, Th\'eor\`eme 2 in Chapitre II and Th\'eor\`eme 1 (c) in Chapitre IV of \cite{Meyer} as well as \cite{J}, \cite{LLQR}, \cite{Str}, \cite{V} or, even more generally, Musielak-Orlicz Hardy spaces \cite{YLK}.   
In this paper we shall mainly focus on certain structural aspects of the space $H^{\log}(\mathbb{R}^d)$ appearing in the work of A. Bonami, S. Grellier, and L. D. Ky  \cite{BGK}.

Before we proceed with the outline of our paper, let us give a formal definition of the space $H^{\log}(\mathbb{R}^d)$. 
Let $\Psi \colon \mathbb{R}^d \times [0, \infty) \rightarrow [0, \infty)$ denote the function given by
$$ \Psi (x,t) : = \frac{t}{ \log (e + t) + \log (e + |x|)}, \quad (x,t) \in \mathbb{R}^d \times [0, \infty ).$$ 
If $B$ is a subset of $\mathbb{R}^d$, one defines $L_{\Psi} (B)$ to be the space of all locally integrable functions $f$ on $B$ satisfying 
$$ \int_B \Psi (x, |f(x)|) dx < \infty.$$

We also fix a non-negative function $\phi \in C^{\infty} (\mathbb{R}^d)$, which is supported in the unit ball of $ \mathbb{R}^d$ with $\int_{\mathbb{R}^d} \phi (y) dy = 1$ and $\phi (x) = c_d$ for all $|x| \leq 1/2$, where $c_d$ is a constant depending on the dimension $d$. Given an $\epsilon > 0$, we employ the notation $ \phi_{\epsilon} (x) : = \epsilon^{-d} \phi(\epsilon^{-1} x)$, $x \in \mathbb{R}^d$. 
\begin{definition}[$H^{\log}$, see \cite{BGK, YLK}]
Let $\phi$ be as above. If $f$ is a tempered distribution on $\mathbb{R}^d$, consider the maximal function 
$$ M_{\phi} (f) (x) := \sup_{\epsilon > 0} | (f \ast \phi_{\epsilon}) (x) |, \quad x \in \mathbb{R}^d.$$

The Hardy space $H^{\log} (\mathbb{R}^d)$ is defined to be the space of tempered distributions $f$ on $\mathbb{R}^d $ such that $ M_{\phi} (f) \in L_{\Psi} (\mathbb{R}^d)$, that is, $M_{\phi}(f)$ satisfies
$$ \int_{\mathbb{R}^d} \Psi (x, M_{\phi} (f) (x)) dx < \infty .$$ 
\end{definition}

The reason for defining $H^{\log} (\mathbb{R}^d)$ comes from the study of products of functions in the real Hardy space $H^1(\mathbb{R}^d)$ and its dual space $ BMO (\mathbb{R}^d)$.
To be more specific, following earlier work by Bonami, T. Iwaniec, P. Jones, and M. Zinsmeister in \cite{BIJZ}, it was shown by Bonami, Grellier, and Ky \cite{BGK} that the product $fg$, in the sense of distributions, of a function 
$f\in H^1(\mathbb{R}^d)$ and a function $g $ of bounded mean oscillation in  $ \mathbb{R}^d$ can in fact be represented as a sum of a continuous bilinear mapping into $L^1(\mathbb{R}^d)$ and a
continuous bilinear mapping into $H^{\log}(\mathbb{R}^d)$. Following \cite{BGK}, for a function $g$ of bounded mean oscillation in $ \mathbb{R}^d $, we set
$$ \| g \|_{BMO^+ (\mathbb{R}^d)} : = \sup_{ \substack{ Q \subset \mathbb{R} : \\ Q \text{ cube}}} \frac{1}{|Q|}  \int_Q | g (x) - \langle g \rangle_Q| dx + \Big| \int_{[0,1)^d} g(x) dx \Big| ,$$
where $\langle g \rangle_Q : = |Q|^{-1} \int_Q g(x) dx$. The aforementioned result of Bonami, Grellier, and Ky can be stated as follows.

\begin{theo}[\cite{BGK}]\label{B-G-K_Thm}
There exist two bilinear operators $S$, $T$ and a constant $C_d > 0$ such that
$$ \| S (f,g) \|_{L^1 (\mathbb{R}^d)} \leq C_d \| f \|_{ H^1 (\mathbb{R}^d ) } \| g \|_{ BMO^+ (\mathbb{R}^d)}  $$
and
$$  \| T (f,g) \|_{H^{\log} (\mathbb{R}^d)} \leq C_d \| f \|_{ H^1 (\mathbb{R}^d) }  \| g \|_{ BMO^+ (\mathbb{R}^d)} $$
with 
$$ f \cdot g = S(f,g) + T(f,g) $$
in the sense of distributions.
\end{theo} 

The operators $S$ and $T$ in the statement of Theorem \ref{B-G-K_Thm} are not unique and they are given in \cite{BGK} in terms of paraproducts that are constructed by using continuous wavelets. See also \cite{Ky},\cite{B_etal}, and \cite{YLK} for further developments. For an introduction to the theory of wavelets, we refer the reader to Y. Meyer's book \cite{Meyer_book}. 

Having seen why $H^{\log} (\mathbb{R}^d)$ is worthy of study, we wish to further elucidate its structure. Our paper consists of three parts.

\subsection*{Part I: Sections \ref{steinproof} and \ref{zygmundproof}} In the first part of this paper we present analogues of the aforementioned theorems of Zygmund and Stein  for $H^{\log} (\mathbb{R}^d)$.
Such results can be derived from more general results previously obtained  in the setting of Orlicz spaces, see for instance \cite{BM,IV}. (We are grateful that these facts were pointed out to us in connection with an earlier note on this subject.) We give a self-contained account here, including a discussion of sharpness, and indicate some minor modifications that need to be made to obtain results in the Musielak-Orlicz setting.

For an $H^{\log}$ version of Stein's theorem, we need to identify the correct analogue of $L\log L$ in this context, which turns out to be $L\log \log L$:  given a measurable subset $B$ of $\mathbb{R}^d$, $L \log \log L (B)$ denotes the class of all locally integrable functions $f$ with $\mathrm{supp}(f) \subseteq B$ and
$$ \int_B |f(x)| \log^+ \log^+ |f(x)| dx < \infty.$$ 

Here is our version of Stein's lemma for $L_{\Psi}$.  
\begin{thm}\label{Stein-type_lemma}
Let $f $ be a measurable function supported in a closed ball $B \subsetneq \mathbb{R}^d$.

Then $M(f) \in L_{\Psi} (B)$ if, and only if, $f \in L \log \log L (B)$.
\end{thm} 
Our proof in fact leads to a more general version of Theorem \ref{Stein-type_lemma}. We discuss this, and give a proof of Theorem \ref{Stein-type_lemma} in Section \ref{steinproof}.

Next is the analogue of Zygmund's result for $H^{\log}(\mathbb{R}^d)$.
\begin{theo}\label{LlogL} Let $B$ denote the closed unit ball in $\mathbb{R}^d$.

If $f$ is a measurable function satisfying $f \in L \log \log L (B)$ and $\int_{B} f (y) dy = 0$, then $f \in H^{\log}(\mathbb{R}^d)$.
\end{theo}
We remark that the mean-zero condition in the hypothesis is in fact necessary in order to place a compactly supported function in $H^{\log} (\mathbb{R}^d) $; see Lemma \ref{1.4_distr} below for a more general version of this fact valid for compactly supported distributions.

\subsection*{Part II: Sections \ref{intro_d}, \ref{d_atom} and \ref{d_BGK}} The second part of this paper covers two different themes. In the first one, we show that one can simplify the argument in \cite{BGK} that establishes Theorem \ref{B-G-K_Thm} by reducing matters to appropriate dyadic counterparts. To be more specific, in Section \ref{intro_d} we introduce a dyadic version of $H^{\log} $ in the periodic setting and then, by establishing a characterisation of dyadic $H^{\log}$ in terms of atomic decompositions, we show that $H^{\log} $ coincides with an intersection of two translates of dyadic $H^{\log}  $, a result of independent interest; see Section \ref{d_atom}. In Section \ref{d_BGK}, we show that, in view of the aforementioned result of Section \ref{d_atom}, one can obtain a simplified proof of Theorem \ref{B-G-K_Thm} in the periodic setting in which only dyadic paraproducts are involved.

\subsection*{Part III: Section \ref{extensions}} The last part of this paper also features two different themes: in Section \ref{periodic_2-3}, we discuss some variants and further extensions of Theorems \ref{Stein-type_lemma} and \ref{LlogL} to the periodic setting. In Section \ref{FC_H^{log}}, we establish a version of a classical inequality of G. H. Hardy and J. E. Littlewood \cite{H-L} that gives a description of the order of magnitude of Fourier coefficients of distributions in $H^{\log} (\mathbb{T})$.



\section{Proof of the Stein-type Theorem for $L_{\Psi}$ and further extensions}\label{steinproof}

We begin with an elementary observation that will be implicitly used several times in the sequel: if $\Phi : [0, \infty) \rightarrow [0, \infty)$ is an increasing function, then for every positive constant $\alpha_0$ one has 
$$ \int_B \Phi (|g (x)| ) dx \leq \Phi (\alpha_0) |B| + \int_{\{ |g| > \alpha_0 \}} \Phi ( |g(x)| ) dx $$ 
for each measurable set $B$ in $\mathbb{R}^d$ with finite measure.

We now turn to the proof of our first theorem. 
\begin{proof}[Proof of Theorem \ref{Stein-type_lemma}]
Assume first that $f \in L \log \log L (B)$. The main observation is that locally the space $L_{\Psi}$ essentially coincides with the Orlicz space defined in terms of the function $\Psi_0 (t) : = t \cdot [\log (e + t)]^{-1}$, $t \geq 0$ and so, one can employ the arguments of Stein \cite{Stein_LlogL}. In view of this observation, we remark that the fact that $f \in L \log \log L (B)$ implies $M (f ) \in L_{\Psi_0} (B)$ is well-known; see for instance \cite[p.242]{BM}, \cite[Sections 4 and 7]{IV}. We shall also include the proof of this implication here for the convenience of the reader.

To be more precise, we note that for $x\in B$ one has
\begin{equation}\label{eq:equivalence}
    \log(e+M(f)(x))\leq \log \brkt{(e+|x|)(e+M(f)(x))}  
\leq c\log(e+M(f)(x)),
\end{equation} 
for a constant $c$ that only depends on $B$.  
Next, an integration by parts yields
$$ \int_{e}^{y} \frac{1}{\log \alpha}d\alpha=\frac{y}{\log y}-e+\int_{e}^{y}\frac{1}{\log^2 \alpha}d\alpha,$$
so that
$$ \frac{y}{\log y}\leq e +\int_{e}^{y}\frac{1}{\log \alpha}d\alpha, \quad \textrm{for} \quad y>e.$$
Together, these two observations imply that
\begin{align*}
 \int_B \Psi (x, M(f)(x) ) dx &\lesssim_B 1 + \int_{ B \cap \{ M(f) > e \} } \Bigg( \int_e^{M(f)(x)} \frac{1}{ \log \alpha } d \alpha \Bigg) dx \\
 &=1 +  \int_e^{\infty} \frac{1}{ \log \alpha } \cdot | \{ x \in B : M (f)(x) > \alpha \} | d \alpha.
\end{align*}
To estimate the last integral, note that there exists an absolute constant $C_d >0$ such that
\begin{equation}\label{wt_sv}
|\{ x \in \mathbb{R}^d : M (f)(x) > \alpha \}| \leq \frac{C_d}{\alpha} \int_{\{ |f| > \alpha/2 \}} |f(x)| dx  
\end{equation}
 for all $\alpha > 0$; see e.g. \cite[(5)]{Stein_LlogL} or Section 5.2 (a) in Chapter I in \cite{Singular}. We thus deduce from \eqref{wt_sv} that
\begin{align*}
\int_B \Psi (x, M(f)(x)) dx &\lesssim_B 1 + \int_B |f(x)|  \Bigg( \int_e^{2|f(x)|} \frac{1}{\alpha \log \alpha} d \alpha \Bigg) dx \\
&\lesssim 1 + \int_{B} |f(x)| \log^+ \log^+ |f(x)| dx,  
 \end{align*}
which implies that $ M (f) \in L_{\Psi} (B)$. 

To prove the reverse implication, assume that for some $f $ supported in $B$ with $f \in L^1 (B)$ we have $M(f) \in L_{\Psi} (B)$. Our task is to show that $f \in L \log \log L (B)$. In order to accomplish this, we shall make use of the fact that there exists a $\rho > 2$, depending only on $\| f \|_{L^1 (B)}$ and $B$, such that we also have $M (f) \in L_{\Psi} (\rho B)$ and moreover, for every $\alpha \geq e^e$,
\begin{equation}\label{wt_reverse}
| \{ x \in \rho B : M (f) (x) > c_1  \alpha \} | \geq \frac{c_2}{\alpha} \int_{B \cap \{ |f| > \alpha \}} |f(x)| dx, 
\end{equation}
where $c_1$, $c_2$ are positive constants that can be taken to be independent of $f$ and $\alpha$. Indeed, arguing as in the proof of \cite[Lemma 1]{Stein_LlogL}, note that for every $r > 2$ one has
\begin{equation}\label{away}
 M (f) (x) \lesssim \frac{1}{(r - 1)^d |B|}  \| f \|_{L^1 (B)}  \quad \mathrm{for} \ \mathrm{all} \ x \in \mathbb{R}^d \setminus r B.
\end{equation}
Hence, if we choose $\rho >2 $ to be large enough, then $ M (f) (x) < e^e \leq \alpha$ for all $x \in \mathbb{R}^d \setminus \rho B$ and so, \eqref{wt_reverse}  follows from \cite[Inequality (6)]{Stein_LlogL}.

Furthermore, one can check that $M(f) \in L_{\Psi} (\rho B)$. Indeed, if we write $B = B (x_0, r_0)$ then, as in \cite{Stein_LlogL}, it follows from the definition of $M$ and the fact that $\mathrm{supp}(f) \subseteq B$ that there exists a constant $c_0>0$, depending only on the dimension, such that for every $ x \in 2B \setminus B$ one has 
\begin{equation}\label{pw_ineq_M}
M (f) (x ) \leq c_0  M (f) \left( x_0 + r_0^2 \cdot \frac{ x-x_0 } { | x - x_0 |^2} \right)
\end{equation}
and so, $M (f) \in L_{\Psi} (2 B)$. To show that \eqref{pw_ineq_M} implies that $M(f) \in L_{\Psi} (B)$, observe first that the function $\Psi_0 $ is increasing on $[0,+\infty)$, and for all $t\geq 1$ and all $s>0$, 
$$
    1\geq \frac{\log(e+s)}{\log(e+ts)}=\frac{\log(e+s)}{\log(e/t+s)+\log t}\geq 
\frac{\log(e+s)}{\log(e+s)+\log t}\geq 
\frac{1}{1+\log t},
$$
so $\Psi_0$ satisfies
\begin{equation}\label{eq:doubling}
        t(1+\log t)^{-1} \Psi_0 (s)\leq \Psi_0 (st)\leq t\Psi_0 (s),
\end{equation} 
which implies that for all $c>0$ and all $s>0$ 
$$ \Psi_0 (cs)\sim_c \Psi_0 (s). $$

Observe that 
a change to polar coordinates, followed by another a change of variables and elementary estimates yield
\begin{align*}
        \int_{2B\setminus B} \Psi_0 (Mf(x)) \dd x 
&\lesssim \int_{r_0}^{2r_0} s^{d-1}\int_{S^{d-1}} \Psi_0 (Mf(x_0+r_0^2 \theta/s))\dd \sigma(\theta)\dd s\\
&\sim r_0^{-1}\int_{\frac{1}{2}}^1  t^{-1-d}\int_{S^{d-1}} \Psi_0 (Mf(x_0+r_0t \theta))\dd \sigma(\theta)\dd t\\
&\sim \int_{\frac{1}{2}}^1  t^{d-1}\int_{S^{d-1}} \Psi_0 (Mf(x_0+r_0t \theta))\dd \sigma(\theta)\dd t \\
& \lesssim \int_B \Psi(x, Mf(x))\dd x.
\end{align*}

Moreover, we deduce from \eqref{away} that $M(f)$ belongs to $ L_{\Psi} (\rho B \setminus 2B)$ and it thus follows that $M (f) \in L_{\Psi} (\rho B)$, as desired. 

Next, note that by the same reasoning as in the proof of sufficiency and by Fubini's theorem,
\begin{align*}
\int_{\rho B} \Psi (x, M (f) (x)) dx &\gtrsim \int_{\rho B \cap  \{ M(f) > \max \{e^e, |x_0|+ r_0 \}  \} } \frac{M(f)(x)}{\log (M (f)(x))} dx \\
& \gtrsim \int_{\rho B \cap\{ M(f) > \max \{e^e, |x_0|+ r_0 \} \} } \Bigg( \int_{e^e}^{M(f)(x)} \frac{1}{ \log \alpha} d \alpha \Bigg) dx \\
& \gtrsim \int_{\max \{e^e, |x_0|+ r_0 \}}^{\infty} \frac{1}{\log \alpha}  | \{x \in \rho B : M(f)(x) > c_2  \alpha \} | d \alpha  .
\end{align*}
By using \eqref{wt_reverse}, we now get
\begin{align*}
\infty  > \int_{\rho B} \Psi (x, M(f)(x)) dx & \gtrsim \int_B |f(x)| \Bigg( \int_{\max \{e^e, |x_0|+ r_0 \}}^{|f(x)|} \frac{1}{\alpha \log \alpha} d \alpha \Bigg) dx \\
&\gtrsim 1 + \int_B |f(x)| \log^+ \log^+ |f(x)| dx 
\end{align*}
and this completes the proof of Theorem \ref{Stein-type_lemma}.
\end{proof}

\begin{rmk} Let $B_0$ denote the closed unit ball in $\mathbb{R}^d$. Given a small $\delta \in (0,e^{-e})$, if, as on pp. 58--59 in \cite{Fefferman}, one considers $f := \delta^{-d} \chi_{\{ |x| < \delta \} }$ then $M (f) (x) \sim |x|^{-d}$ for all $|x| > 2 \delta$ and so,
\begin{equation}\label{equivalence_ex} 
\int_{B_0} |f(x)| \log^+ \log^+ |f (x)| dx \sim \log (\log (\delta^{-1}) ) \sim \int_{B_0} \Psi (x, M(f) (x)) dx .
\end{equation}
This shows that given $L_{\Psi} (B_0)$, the space $L \log \log L (B_0)$ in the statement of Theorem \ref{Stein-type_lemma} is best possible in general, in terms of size. 

Indeed, the left-hand side of \eqref{equivalence_ex} follows by direct calculation. On the other hand,  \eqref{eq:equivalence}, \eqref{eq:doubling}, a change to polar coordinates, and further change of variables yield
 \begin{align*}
\int_{B_0} \Psi (x, M(f) (x)) dx &\sim 1+\int_{2\delta}^{1} \frac{1}{\log(e+s^{-d})}\frac{\dd s}{s}\\
&\sim 1+\int_{1}^{(2\delta)^{-1}} \frac{1}{\log(e+u^d)}\frac{\dd u}{u}\sim 1+\int_{e}^{(2\delta)^{-1}} \frac{1}{\log(u)}\frac{\dd u}{u},
 \end{align*}
from where the right-hand side of \eqref{equivalence_ex} follows.
\end{rmk}

\subsection{Further generalisations} Assume that $\Psi : \mathbb{R}^d \times [0, \infty)$ is a non-negative function satisfying the following properties:
\begin{enumerate}
\item For every $x \in \mathbb{R}^d$ fixed, $\Psi (x,t) = \Psi_x (t)$ is \emph{Orlicz} in $t \in [0,\infty)$, namely $\Psi_x (0) = 0$, $\Psi_x$ is increasing on $[0,\infty)$ with $\Psi_x (t) > 0$ for all $t >0$ and $\Psi_x (t) \rightarrow \infty$ as $t \rightarrow \infty$. 

Moreover, assume that there exists an absolute constant $C_0 > 0$ such that $\Psi_x (2t) \leq C_0 \Psi_x (t) $ for all $x \in \mathbb{R}^d$ and every $t \in [0, \infty) $. 
\item If $K$ is a compact set in $\mathbb{R}^d$, then there exist $x_1, x_2 \in K$ and a constant $C_K > 0$ such that
$$ C_K^{-1} < \Psi (x_1, t) \leq \Psi (x,t) \leq \Psi (x_2,t)  < C_K $$
for every $x \in K$ and for all $t >0$. 
\item If we write $\Psi (x,t) = \Psi_x (t) = \int_0^t \psi_x (s) ds$, then for every $\alpha_0$, $\beta_0$ with $0 < \alpha_0 < \beta_0 $ one has
$$ \int_{\alpha_0}^{\beta_0} \frac{\psi_x (s)}{s} ds < \infty$$
for every $x \in \mathbb{R}^d$. 
\end{enumerate}

By carefully examining the proof of Theorem \ref{Stein-type_lemma}, one obtains the following result.

\begin{theo} \label{generalstein}

Let $\Psi (x,t) = \int_0^t \psi_x (s) ds$, $(x,t) \in \mathbb{R}^d \times [0, \infty)$, be as above.

Fix a closed ball $B$ with $B \subsetneq \mathbb{R}^d$ and let $f$ be such that $\mathrm{supp} (f) \subseteq B$.  Then, $M (f ) \in L_{\Psi} (B)$ if, and only if, 
$$ \int_{\{ | f| > \alpha_0 \} } | f(x) |  \Bigg( \int_{\alpha_0}^{|f(x)|} \frac{\psi_x (s)}{s} ds \Bigg) dx < \infty $$
for every $\alpha_0 > 0$.
\end{theo}

Theorem \ref{generalstein} applies to certain Orlicz spaces considered in connection with convergence of Fourier series, see e.g. \cite{Antonov, Sjolin}, and the recent paper by V. Lie \cite{Lie}; we give some sample applications in Subsection \ref{applications_periodic}. 


\section{Proof of the Zygmund-type Theorem for $H^{\log} (\mathbb{R}^d)$}\label{zygmundproof}
We begin with the following elementary lemmas.

\begin{lemma}\label{decreasing} Consider the function $g : [0, \infty )^2 \rightarrow [0, \infty)$ given by
$$ g (s,t) : = \frac{1}{ \log (e + t) + \log (e + s)}, \quad (s,t) \in [0, \infty )^2. $$
Then one has
$$ \Psi (x,t) \leq \int_0^t g(|x| , \tau) d \tau \leq 2 \Psi (x,t) $$ 
for all $(x,t) \in \mathbb{R}^d \times [0, \infty)$.
\end{lemma}
\begin{proof}
The function $t\mapsto g(s,t)=[ \log \big( (e+t)(e+s) \big)]^{-1}$ is decreasing, so clearly 
$$	\int_0^t g(|x|,s)\dd s\geq tg(|x|,t)=\Psi(x,t). $$ 
We now address the upper bound. A calculation yields that
$$ \partial_t (t^\epsilon g(|x|,t))=\frac{t^{\epsilon}}{\log (e+t)+\log(1+|x|)}\Bigg(\frac{\epsilon}{t}-\frac{1}{(e+t)[\log(e+t)+\log(e+|x|) ]}\Bigg), $$
and we observe that the term within the parenthesis is positive if, and only if,
$$ \frac{\epsilon}{t}-\frac{1}{(e+t)[\log(e+t)+\log(e+|x|) ]}>0, $$
which for $\epsilon=1/2$ is equivalent to the inequality 
$$ (e+t)[\log(e+t)+\log(e+|x|)]>2t.	$$
But clearly 
$$ (e+t)[\log(e+t)+\log(e+|x|)]\geq 2(e+t)>2t. $$
Thus $s\mapsto s^\epsilon g(|x|,s)$ is increasing for $\epsilon =1/2$, which implies that 
$$ \int_0^t g(|x|,s)d s=\int_0^t s^{-\epsilon} s^{\epsilon}g(|x|,s) d s\leq \frac{1}{1-\epsilon}\Psi(x,t)=2\Psi(x,t) $$
and this completes the proof of the lemma.
\end{proof}

\begin{lemma}\label{translation}
Let $x_0 \in \mathbb{R}^d$ be fixed and for $u \in S(\mathbb{R}^d)$ define $\langle \tau_{x_0}f , u \rangle := \langle f, \tau_{-x_0}u \rangle$, where $\tau_{-x_0} u (x)  := u (x-x_0)$, $x \in \mathbb{R}^d$. 

Then $f\in H^{\log}(\mathbb{R}^d)$ if, and only if, $\tau_{x_0}f \in H^{\log}(\mathbb{R}^d)$.
\end{lemma}

\begin{proof} Note that it suffices to prove that for any  $x_0 \in \mathbb{R}^d$ and $f \in H^{\log}(\mathbb{R}^d)$ one also has that $\tau_{x_0}f \in H^{\log}(\mathbb{R}^d)$. 

Towards this aim, fix an $x_0 \in \mathbb{R}^d$ and an $f \in H^{\log} (\mathbb{R}^d)$. Observe that, by using a change of variables and the translation invariance of $M_{\phi}$, we may write 
$$ I: = \int_{\mathbb{R}^d} \frac{M_{\phi} (\tau_{x_0} f ) (x) }{ \log (e + |x|) + \log (e + M_{\phi} (\tau_{x_0} f) (x) ) } dx 
$$ as
$$ I = \int_{\mathbb{R}^d} \frac{M_{\phi} ( f ) (x ) }{ \log (e + |x-x_0|) + \log (e + M_{\phi} ( f) (x) ) } dx.  $$
To prove that $I < \infty$, we split
$$ I = I_1 + I_2 ,$$
where
$$ I_1 := \int_{ |x| > 4 |x_0| } \frac{M_{\phi} ( f ) (x ) }{ \log (e + |x-x_0|) + \log (e + M_{\phi} ( f) (x) ) } dx $$
and
$$ I_2 := \int_{ |x| \leq 4 |x_0| } \frac{M_{\phi} ( f ) (x ) }{ \log (e + |x-x_0|) + \log (e + M_{\phi} ( f) (x) ) } dx . $$
To show that $I_1 < \infty$, observe that for $|x| > 4 |x_0|$ one has
$$ \frac{4 |x-x_0| }{5} < |x| < \frac{4 |x-x_0| } {3} $$
and so,
\begin{align*}
 I_1 & \lesssim  \int_{ |x| > 4 |x_0| } \frac{M_{\phi} ( f ) (x ) }{ \log (e + |x|) + \log (e + M_{\phi} ( f) (x) ) } dx \\
& \leq \int_{ \mathbb{R}^d } \frac{M_{\phi} ( f ) (x) }{ \log (e + |x|) + \log (e + M_{\phi} ( f) (x) ) } dx .
\end{align*}
Since $f \in H^{\log} (\mathbb{R}^d)$, the last integral is finite and we thus deduce that $I_1 < \infty$.  Next, to show that $I_2 < \infty$, we have 
\begin{align*}
I_2 & \leq \int_{ |x| \leq 4 |x_0| } \frac{M_{\phi} ( f ) (x ) }{ 1 + \log (e + M_{\phi} ( f) (x) ) } dx \\
& \lesssim_{|x_0|} \int_{ |x| \leq 4 |x_0| } \frac{M_{\phi} ( f ) (x) }{ \log (e + |x|) + \log (e + M_{\phi} (f) (x) ) } dx \\
& \leq  \int_{ \mathbb{R}^d } \frac{M_{\phi} (f) (x ) }{ \log (e + |x|) + \log (e + M_{\phi} (f) (x) ) } dx
\end{align*}
and so, $I_2 < \infty$, as $f \in H^{\log} (\mathbb{R}^d)$. Therefore, $I < \infty$ and it thus follows that $\tau_{x_0} f \in H^{\log} (\mathbb{R}^d )$. \end{proof}

To obtain the desired variant of Zygmund's theorem, we shall use the fact that functions in $H^{\log} (\mathbb{R}^d)$ have mean zero; see  Lemma 1.4 in \cite{BIJZ}. One can actually establish the following more general fact.
\begin{lemma}\label{1.4_distr} If $f \in H^{\log} (\mathbb{R}^d)$ is a compactly supported distribution, then $\widehat{f}(0)= 0$. 
\end{lemma}

\begin{proof} Let $f$ be a compactly supported distribution in $ H^{\log} (\mathbb{R}^d)$. In light of Lemma \ref{translation}, we may assume, without loss of generality,  that $f$ is supported in a closed ball $B_r$ centered at $0$ with radius  $r >0$, i.e. $\mathrm{supp} (f) \subseteq B_r := \{ x \in \mathbb{R}^d : |x|< r \}$. 

Let $\chi \in C_c^\infty(\mathbb{R}^d)$ be supported inside $B_r$ and be equal to $1$ on the support of $f$.

To prove the lemma, take an $x \in \mathbb{R}^d$ with $|x| > 2r$ and observe that, by the definition of $\phi_{\epsilon}$, if we take $\epsilon=4|x|$ we have that
$$
    | f \ast \phi_{\epsilon} (x) |=\frac{1}{\epsilon^d}\big|\langle f,\chi \phi(\epsilon^{-1}(\cdot-x)) \rangle \big| \gtrsim \frac{1}{|x|^d} |\langle f,\chi \rangle|
$$
as we then have $\phi(\epsilon^{-1}(x-y))=c_d$ for $y\in B_r$. Notice that $|\langle f,1-\chi \rangle|=0$.  Therefore, for all $|x| > 2r$ and $\epsilon=4|x|$, we have
$$ 
    M_{\phi} (f) (x) \gtrsim \frac{1}{|x|^d}  | \langle f,1 \rangle |,
$$
and so, we deduce from Lemma \ref{decreasing} that
$$ \Psi (x, M_{\phi}(f) (x)) \gtrsim \frac{1}{|x|^d \log (e + |x|)}  | \widehat{f}(0)| $$
for $|x|$ large enough.

Hence, if $\widehat{f}(0)\neq 0$, then the function $\Psi (x, M_{\phi} (f)(x)) $ does not belong to $L^1 (\mathbb{R}^d)$, which is a contradiction.
\end{proof}

We are now ready to prove Theorem \ref{LlogL}.



\begin{proof}[Proof of Theorem \ref{LlogL}] Let $B$ denote the closed unit ball in $\mathbb{R}^d$. Fix a function $f$ with $\mathrm{supp} (f) \subseteq B $, $\int_B f (y) dy = 0$ and $f \in L \log \log L (B )$. First of all, observe that 
$$ M_{\phi} (f) (x) \lesssim M(f)(x) \quad \mathrm{for}\ \mathrm{all}\ x \in \mathbb{R}^d, $$
where $M(f)$ denotes the Hardy-Littlewood maximal function of $f$; see e.g. Theorem 2 on pp. 62--63 in \cite{Singular}. We thus deduce from Lemma \ref{decreasing} that
$$  \Psi (x, M_{\phi} (f)(x)) \lesssim \Psi (x, M(f)(x))  \quad \mathrm{for}\ \mathrm{all}\ x \in \mathbb{R}^d $$
and hence, by using Theorem \ref{Stein-type_lemma}, we obtain
\begin{equation}\label{local_bound}
 \int_{2 B } \Psi (x, M_{\phi}(f)(x)) dx \lesssim 1 + \int_{2 B } |f(x)| \log^+ \log^+ |f(x)| dx,
\end{equation}
where $2 B  : = \{x \in \mathbb{R}^d : |x| \leq 2\} $.

To estimate the integral of $\Psi (x, M_{\phi} (f) (x))$ for $x \in \mathbb{R}^d \setminus 2 B $, we shall make use of the cancellation of $f$. To be more specific, observe that if $|x|>2$ then for every $\epsilon < |x|/2$, one has that
$$f \ast \phi_{\epsilon}(x) = \frac{1}{\epsilon^d} \int_{B } f(y) \phi \Big( \frac{x-y}{\epsilon} \Big) dy = 0$$
since $|x-y|/\epsilon > 1$ whenever $y \in B $. Therefore, we may restrict ourselves to $\epsilon \geq |x|/2$ when $|x|>2$. Hence, for $\epsilon \geq |x|/2$, by exploiting the cancellation of $f$ and using a Lipschitz estimate on $\phi_{\epsilon}$, we obtain
\begin{align*}
|f \ast \phi_{\epsilon} (x)| &= \frac{1}{\epsilon^d} \Bigg| \int_B f(y)  \phi \Big( \frac{x-y}{\epsilon} \Big) dy \Bigg| = \frac{1}{\epsilon^d} \Bigg| \int_B f(y) \Big[  \phi \Big( \frac{x-y}{\epsilon} \Big) - \phi \Big( \frac{x}{\epsilon} \Big) \Big] dy \Bigg| \\
&\lesssim_{\phi} \frac{1}{\epsilon^{d+1}} \int_B |y \cdot f(y)| dy \lesssim \frac{1}{|x|^{d+1} } \Bigg[ 1 + \int_B |f(y)| \log^+ \log^+ |f(y)| dy \Bigg].
\end{align*}
We thus deduce that, for every $x \in \mathbb{R}^d \setminus 2 B $,
$$ |M_{\phi} (f) (x)| \lesssim \frac{1}{|x|^{d+1}} \Bigg[1 + \int_B |f(y)| \log^+ \log^+ |f(y)| dy \Bigg] $$
and so, 
\begin{align*}
& \int_{\mathbb{R}^d \setminus 2 B } \Psi (x, M_{\phi} (x)) dx  \\
&  \lesssim \Bigg[ 1 + \int_B |f(y)| \log^+ \log^+ |f(y)| dy \Bigg] \int_{\mathbb{R}^d \setminus 2 B } \frac{1} {|x|^{d+1 }\log (e +|x|)} dx \\
& \lesssim 1 + \int_B |f(y)| \log^+ \log^+ |f(y)| dy,
\end{align*}
as desired. Therefore, Theorem \ref{LlogL} is now established by using the last estimate combined with \eqref{local_bound}.
\end{proof}

\subsection{A partial converse}\label{periodic}

As in the classical setting of the real Hardy space $H^1$, see \cite{Stein_LlogL}, Theorem \ref{LlogL} has a partial converse. To be more precise, if a function $f$ is positive on an open set $U$ and $f$ belongs to $ H^{\log} (\mathbb{R}^d)$, then the function $f \in L \log \log L (K)$ for every compact set $K\subset U$. 

Indeed, to see this, note that if $f$ is as above then
$$ M_{\phi} (f) (x) \gtrsim M (f \cdot \eta_K ) (x) \quad \mathrm{for}\ \mathrm{all} \ x \in K,$$
where $\eta_K$ is an appropriate Schwartz function with $\eta_K \sim 1$ on $K$; see e.g. Section 5.3 in Chapter III in \cite{Big_Stein}. Hence, by using Lemma \ref{decreasing} and Theorem \ref{Stein-type_lemma}, we get
\begin{align*}
 \int_{\mathbb{R}^d}  \Psi (x, M_{\phi} (f) (x) ) dx \geq  \int_K  \Psi (x, M_{\phi} (f) (x) ) dx &\gtrsim \int_K \Psi (x, M (\eta_K \cdot f) (x) ) dx \\
&  \gtrsim  1 +   \int_K |f(x)| \log^+ \log^+ |f(x)| dx. 
\end{align*}


\section{Dyadic $H^{\log}$ on $\mathbb{T}$}\label{intro_d}
 
In this section we introduce a dyadic variant of $H^{\log} (\mathbb{T})$ and in the next section we shall prove that it admits a characterisation in terms of atomic decompositions. Here, we adopt the convention that  $\mathbb{T} \cong \mathbb{R}/2\pi \mathbb{Z}$. 

Before we proceed, let us recall that, following \cite{BIJZ}, $H^{\log} (\mathbb{T})$
is defined as the class of all $f \in \mathcal{D'}$ whose non-tangential maximal function $f^{\ast}$ satisfies
$$\int_0^{2\pi} \Psi_0  (  | f^{\ast}  ( \theta)|  ) d \theta < \infty , $$
where $\Psi_0$ is as above; $\Psi_0 (t) : = t \cdot [\log (e +t)]^{-1}$, $t \geq 0$.  Here, $\mathcal{D}'$ denotes the class of all distributions on $\mathbb{T}$. For $f \in H^{\log} (\T)$, one sets 
$$ \| f \|_{H^{\log} (\mathbb{T})} :  = \inf \Bigg\{ \lambda > 0 : (2\pi)^{-1} \int_0^{2\pi} \Psi_0  (  \lambda^{-1}| f^{\ast}  ( \theta)|  ) d \theta \le 1  \Bigg\} . $$

In what follows, for $f \in \mathcal{D}'$ and $n \in \mathbb{Z}$, we write $\widehat{f} (n) : = \langle f, e_n \rangle$ where $e_n (\theta) : = e^{i n \theta}$, $\theta \in [0, 2 \pi)$. 

\begin{rmk}\label{p-LlogL}
Recall that for $p>0$, the real Hardy space $H^p (\mathbb{T})$ is defined as the class of all $f \in \mathcal{D}'$ such that $f^{\ast} \in L^p (\mathbb{T})$.

It is well-known that for $p \geq 1$, elements in $H^p (\mathbb{T})$ are functions and moreover, $H^p (\mathbb{T}) \equiv L^p (\mathbb{T})$ for $p > 1$.  Clearly, 
$$ H^1 (\mathbb{T}) \subset H^{\log} (\mathbb{T}) \subset H^p (\mathbb{T})  \quad \text{for all } 0 < p < 1. $$
\end{rmk}

\subsection{Definition of dyadic $H^{\log}$ on $\mathbb{T}$}
Let $\mathcal{I}$ be a given system of dyadic arcs in $\mathbb{T}$. In particular, following the formulation of \cite{Mei}, we assume that $\mathcal{I}$ is of the form
$$  \mathcal{I}^{\delta} := \big\{ \big[ 2 \pi (2^{-k} m + \delta ) , 2 \pi ( 2^{-k} (m+1) + \delta )  \big) : k \in \mathbb{N}_0, m =  0, \cdots, 2^k-1 \big\}  $$
for some  $ \delta \in [0,1) $.

If $I \in \mathcal{I} $, then $h_I$ denotes the cancellative Haar function associated to $I$ that is,
$$ h_I := | I |^{-1/2} \big( \chi_{I_-} - \chi_{I_+} \big) ,$$
where $I_-$ and $I_+$ denote the left and right halves of $I$, respectively. 
 
Here, we shall adopt the following convention: if $\mathbf{f} = \{ f_I \}_{I \in \mathcal{I}} \cup \{ f_0 \}$ is a collection of complex numbers, we consider the associated sequence of functions $ (f_N)_{N \in \mathbb{N}}$ given by
$$ f_N (\theta) : = f_0 + \sum_{\substack{I \in \mathcal{I} : \\ |I| \geq 2^{-N}}} f_I h_I (\theta), \quad \theta \in \mathbb{T} .  $$
As usual, the dyadic square function $S_{\mathcal{I}} [f_N]$ of $f_N$ is given by
$$ S_{\mathcal{I}} [f_N] ( \theta ) : = |f_0| +  \Big( \sum_{\substack{I \in \mathcal{I} : \\ |I| \geq 2^{-N}}} | f_I |^2 \frac{\chi_I (\theta )} {| I |} \Big)^{1/2}, \quad \theta \in \mathbb{T} . $$
One then defines the dyadic square function $S_{\mathcal{I}} [\mathbf{f}]$ of $\mathbf{f}$ as the pointwise limit
$$ S_{\mathcal{I}} [\mathbf{f}] := \lim_{N \rightarrow \infty} S_{\mathcal{I}} [f_N] =  |f_0| +  \Big( \sum_{ I \in \mathcal{I}  } | f_I |^2 \frac{\chi_I} {| I |} \Big)^{1/2}. $$
 
\begin{definition} We define $h^{\log}_{\mathcal{I}} (\mathbb{T})$ as the class of all collections $\mathbf{f} = \{ f_I \}_{I \in \mathcal{I}} \cup \{ f_0 \}$ of complex numbers satisfying
$$ \int_{\mathbb{T}} \Psi_0 \big( S_{\mathcal{I}} [ \mathbf{f}] ( \theta ) \big) d \theta < \infty .$$
If $\mathbf{f} \in h^{\log}_{\mathcal{I}} (\mathbb{T})$, we set
$$ \| \mathbf{f} \|_{h^{\log}_{\mathcal{I}} (\mathbb{T}) } : = \inf \Bigg\{ \lambda > 0 : (2\pi)^{-1}  \int_{\mathbb{T}} \Psi_0 \big( \lambda^{-1}  S_{\mathcal{I}} [ \mathbf{f}] ( \theta ) \big) d \theta \leq 1 \Bigg\}.  $$

\end{definition}


\subsubsection{Some remarks}\label{remarks}

It can easily be seen that there exists an absolute constant $C_0>1$ such that for all $\mathbf{f}, \mathbf{g} \in h^{\log}_{\mathcal{I}} (\mathbb{T})$ and $\mu \in \mathbb{C}$ one has 
$$ \| \mathbf{f} + \mathbf{g} \|_{ h^{\log}_{\mathcal{I}} (\mathbb{T}) } \leq C_0 (\|  \mathbf{f} \|_{ h^{\log}_{\mathcal{I}} (\mathbb{T}) )} +  \| \mathbf{g} \|_{ h^{\log}_{\mathcal{I}} (\mathbb{T}) }) $$ 
and
$$ \| \mu \mathbf{f } \|_{ h^{\log}_{\mathcal{I}} (\mathbb{T}) } = |\mu| \| \mathbf{f } \|_{ h^{\log}_{\mathcal{I}} (\mathbb{T}) }.$$
Moreover, $\mathbf{f} = \mathbf{0}$ if, and only if, $\| \mathbf{f} \|_{h^{\log}_{\mathcal{I}}} (\mathbb{T}) = 0$.
In particular, $\| \cdot \|_{h^{\log}_{\mathcal{I}} (\mathbb{T})}$ is a quasi-norm on the linear space $h^{\log}_{\mathcal{I}} (\mathbb{T}) $ and one can show that $(h^{\log}_{\mathcal{I}} (\mathbb{T}), \| \cdot \|_{h^{\log}_{\mathcal{I}} (\mathbb{T})})$ is complete.

Let $F_{\mathcal{I}} (\mathbb{T})$ denote the class of all functions $f \in L^1 (\mathbb{T})$ such that the collection $\{ \langle f, h_I \rangle \}_{I \in \mathcal{I}}\cup \{ \widehat{f}(0) \} $ consists of finitely many non-zero terms. Note that if $f \in F_{\mathcal{I}} (\mathbb{T}) $ one can write $f = \widehat{f}(0)+ \sum_{I \in \mathcal{I}} \langle f, h_I \rangle h_I$ and moreover, by identifying functions in $F_{\mathcal{I}} (\mathbb{T})$ with the corresponding collections of their Haar coefficients, we may regard $F_{\mathcal{I}} (\mathbb{T})$ as a dense subspace of $h^{\log}_{\mathcal{I}} (\mathbb{T}) $. 


\subsection{$h^{\log}_{\mathcal{I}} (\mathbb{T})$ and $H^{\log}_{\mathcal{I}} (\mathbb{T})$}\label{prelim}

Our goal in this section is to show that every collection in $h^{\log}_{\mathcal{I}} (\mathbb{T})$ can be regarded in a `canonical' way as an element of $ \mathcal{D}' $. More specifically, we shall prove that if $\mathbf{f} \in h^{\log}_{\mathcal{I}} (\mathbb{T})$, then the corresponding sequence of functions $(f_N)_{N \in \mathbb{N}}$ in $F_{\mathcal{I}} (\mathbb{T})$ converges in the sense of distributions to some $f \in \mathcal{D}' $. If $\iota$ denotes the associated map from $h^{\log} (\T)$ to $\mathcal{D}'$, then one defines $H^{\log}_{\mathcal{I}} (\mathbb{T}):= \iota \big[  h^{\log}_{\mathcal{I}} (\mathbb{T}) \big] $ and $\| f \|_{H^{\log}_{\mathcal{I}} (\mathbb{T})} := \| \mathbf{f} \|_{ h^{\log}_{\mathcal{I}} (\mathbb{T})} $ for $\mathbf{f} \in h^{\log}_{\mathcal{I}} (\mathbb{T})$ with $f = \iota ( \mathbf{f} )$.

To this end, let $\mathbf{f} = \{ f_I \}_{I \in \mathcal{I}} \cup \{ f_0 \}$ be a given collection in $h^{\log}_{\mathcal{I}} (\mathbb{T})$. It can easily be seen that for every $p \in (1/2,1)$ one has    
\begin{equation}\label{ineq_el}
\Psi_0 (t) \geq (1-p) t^p  \quad \text{ for all } t \geq 0 . 
\end{equation}
Hence, by using \eqref{ineq_el}, we deduce that
\begin{equation}\label{implication_p}
\| S_{\mathcal{I}} [ \mathbf{f} ]  \|_{L^p (\mathbb{T})} \leq D(p,\mathbf{f}),
\end{equation}
where $D(p,\mathbf{f}) $ is a (finite) positive constant given by
\begin{equation}\label{D(f,p)}
D(p,\mathbf{f}) : = (p-1)^{-1/p} \Big( \int_{\mathbb{T}} \Psi_0 ( S_{\mathcal{I}} [\textbf{f}] ( \theta ) ) d \theta \Big)^{1/p} .
\end{equation}
Since
$$ S_{\mathcal{I}} [ \mathbf{f} ]  \geq \Bigg( \sum_{ \substack { I \in \mathcal{I} : \\ |I| = 2^{-k}}  } | f_I |^2 \frac{\chi_I} {| I |} \Bigg)^{1/2} \quad \text{for all } k \in \mathbb{N}_0, $$
we deduce from \eqref{implication_p} that
$$
D(p,\mathbf{f}) \geq \Bigg(\int_{\mathbb{T}} \Big(  \sum_{ \substack { I \in \mathcal{I} : \\  |I| = 2^{-k}}  } | f_I |^2 \frac{\chi_I ( \theta )} {| I |}  \Big)^{p/2} d \theta \Bigg)^{1/p}  \quad \text{for all } k \in \mathbb{N}_0.  $$
We thus have
\begin{equation}\label{lower_est}
D(p,\mathbf{f}) \geq  \Bigg( \sum_{\substack { I \in \mathcal{I} : \\ |I| = 2^{-k}}  }  | f_I |^p  | I |^{1-p/2}   \Bigg)^{1/p} \quad \text{for all } k \in \mathbb{N}_0,
\end{equation}
where we used the fact that the arcs $I \in \mathcal{I}$ with $|I| = 2^{-k}$ are mutually disjoint. We shall combine \eqref{lower_est} with the following standard estimate.

\begin{lemma}\label{est} There exists an absolute constant $C_0>0$ such that
$$ | \langle \phi, h_I \rangle | \leq C_0 \| \phi'\|_{L^{\infty} (\mathbb{T})} |I |^{3/2} $$
for all  $\phi \in C^{1} (\mathbb{T})$ and for all $I \in \mathcal{I}$.  
\end{lemma}

\begin{proof} We may assume without loss of generality that $I \subset \mathbb{T}$ can be regarded as an interval in $[0,2 \pi )$. By using a change of variables and the mean value theorem, we have
\begin{align*}
 | \langle \phi, h_I \rangle | & = | I |^{-1/2} \Big| \int_{I_-} \phi ( \theta ) d \theta - \int_{I_+} \phi ( \theta ) d \theta  \Big|  \\
 & =  | I |^{-1/2} \Big| \int_{I_-} [  \phi ( \theta )  -   \phi ( \theta + |I|/2 ) ]  d \theta  \Big| \\
 & \lesssim \| \phi'\|_{L^{\infty} (\mathbb{T})} |I |^{3/2},
\end{align*}
as desired.
\end{proof}

As mentioned above, we shall prove that the sequence of functions $(f_N)_{N \in \mathbb{N}}$  associated to $\mathbf{f}$ converges in the sense of distributions. Towards this aim, by using Lemma \ref{est}, for $N, M \in \mathbb{N}$ with $N>M$, we have
\begin{align*}
 | \langle f_N - f_M, \phi \rangle | = \Bigg| \sum_{\substack { I \in \mathcal{I} : \\  2^{-M} > |I| \geq 2^{-N}}  } f_I \langle h_I , \phi \rangle \Bigg|  
 & \leq \sum_{k=M+1}^N \sum_{\substack { I \in \mathcal{I} : \\  |I| =  2^{-k}}  } | f_I | | \langle  \phi , h_I  \rangle   | \\
&  \lesssim \| \phi' \|_{L^{\infty} (\mathbb{T})} \sum_{k=M+1}^N   \sum_{\substack { I \in \mathcal{I} : \\  |I| =  2^{-k}}  } | f_I | | I |^{3/2} .
\end{align*}
Fix a $p \in (1/2, 1)$ and note that the previous estimate implies that 
\begin{equation}\label{est_p}
| \langle f_N - f_M, \phi \rangle | \lesssim \| \phi' \|_{L^{\infty} (\mathbb{T})} \sum_{k=M+1}^N \Bigg( \sum_{\substack { I \in \mathcal{I} : \\  |I| =  2^{-k}}  } | f_I |^p | I |^{3p/2} \Bigg)^{1/p} . 
\end{equation}
Since
\begin{align*}
  \sum_{k=M+1}^N \Bigg( \sum_{\substack { I \in \mathcal{I} : \\ |I| =  2^{-k}} } | f_I |^p | I |^{3p/2} \Bigg)^{1/p} & = \sum_{k=M+1}^N \Bigg( \sum_{\substack { I \in \mathcal{I} : \\ |I| =  2^{-k}} } | f_I |^p |I|^{1- p/2} |I|^{2p - 1}   \Bigg)^{1/p}  \\
  & = \sum_{k=M+1}^N  2^{- (2 -1/p) k} \Bigg(   \sum_{\substack { I \in \mathcal{I} : \\  |I| =  2^{-k}}  } | f_I |^p |I|^{1- p/2} \Bigg)^{1/p},
\end{align*}
it follows from \eqref{lower_est} that
\begin{equation}\label{upper_N}
 \sum_{k=M+1}^N \Bigg(  \sum_{\substack { I \in \mathcal{I} : \\  |I| =  2^{-k}}  } | f_I |^p | I |^{3p/2} \Bigg)^{1/p}   \lesssim   \Bigg( \sum_{k=M+1}^N  2^{- (2 -1/p) k} \Bigg) D(p,\mathbf{f}).
 \end{equation}
Therefore, \eqref{est_p} and \eqref{upper_N} imply that $(f_N)_{N \in \mathbb{N}}$ is Cauchy in $\mathcal{D}'$ and so, it converges to some $f \in \mathcal{D}'$.
A completely analogous argument shows that
$$  \sup_{N \in \mathbb{N}} | \langle f_N, \phi \rangle | \leq |f_0| \| \phi \|_{L^{\infty}(\mathbb{T})} + c_p  D(p,  \mathbf{f} )   \| \phi' \|_{L^{\infty} (\mathbb{T})} \quad \text{for all } \phi \in C^1 (\mathbb{T}), $$
where $c_p > 0 $ is a constant that depends only on $p$. Hence,  
\begin{equation}\label{sup_N}
| \langle f , \phi \rangle | \leq |f_0| \| \phi \|_{L^{\infty} (\mathbb{T})} + c_p  D(p,  \mathbf{f} )   \| \phi' \|_{L^{\infty} (\mathbb{T})} \quad \text{for all } \phi \in C^1 (\mathbb{T}). 
\end{equation}

 
\section{Atomic decomposition of $H^{\log}_{\mathcal{I}} (\mathbb{T})$}\label{d_atom}

\subsection{Atomic decomposition of $H^{\log} (\mathbb{T})$}

 In \cite{Ky}, Ky showed that $H^{\log} (\mathbb{R}^d)$ admits a decomposition in terms of $H^{\log} (\mathbb{R}^d)$-atoms, see also Chapter 1 in \cite{YLK}. An adaptation of Ky's argument to the periodic setting establishes an analogous result in the periodic setting. 
 
 In order to state the characterisation of  $H^{\log} (\mathbb{T})$ in terms of atomic decompositions, we need the following definitions.

\begin{definition}[\cite{Ky}]\label{H^log_atoms}
Let $I \subset \mathbb{T}$ be an arc. A measurable function $a_I$ on $\mathbb{T}$ is said to be an $H^{\log} (\mathbb{T})$-atom associated to $I$ whenever
\begin{itemize}
\item $\mathrm{supp} (a_I) \subseteq I$,
\item $\int a_I (\theta) d \theta =0$,  and
\item  $\| a_I \|_{L^{\infty} (\mathbb{T})} \leq \| \chi_I \|_{L^{\log} (\mathbb{T})} $.
\end{itemize}
\end{definition}

Following \cite{Ky, YLK}, see also \cite{V}, we give the definition of atomic $H^{\log} (\mathbb{T}) $.

\begin{definition}\label{H^log} The atomic Hardy-Orlicz space $H^{\log}_{\mathrm{ at } } (\mathbb{T})$ is defined as the space of all $f \in \mathcal{D}' $ that can be written as 
$$ f - \widehat{f} (0) = \sum_{k \in \mathbb{N}} b_{I_k} \quad \text{in } \mathcal{D}',$$
where $b_{I_k} = \mu_k a_{I_k}$ with $a_{I_k} $ being an atom in $H^{\log} (\mathbb{T})$ associated to some arc $I_k$ and $\{ \mu_k \}_{k \in \mathbb{Z}}$ is a collection of complex scalars such that
$$ \sum_{k \in \mathbb{N}} | I_k | \Psi_0 ( \| b_{I_k} \|_{L^{\infty} (\mathbb{T}) } ) < \infty. $$ 
If $\{ b_{I_k} \}_{k \in \mathbb{N}}$ is as above, let
$$ \Lambda_{\infty} \big( f, \{ b_{I_K} \}_{k \in \mathbb{N}} \big) := \inf \Bigg\{ \lambda>0 :  \Psi_0 (\lambda^{-1} |\widehat{f} (0)|) + \sum_{k \in \mathbb{N} } | I_k | \Psi_0 (  \lambda^{-1} \| b_{I_k} \|_{L^{\infty} (\mathbb{T})})  \leq 1 \Bigg\}  $$
and define
$$ \| f \|_{H^{\log}_{ \mathrm{ at} } (\mathbb{T})} := \inf \Bigg\{ \Lambda_{\infty} \big( f, \{ b_{I_k} \}_{k \in \mathbb{N}} \big) : f - \widehat{f} (0)=  \sum_{k \in \mathbb{N} } b_{I_k} \ \mathrm{with} \ \{b_{I_k}\}_{k \in \mathbb{N} } \ \mathrm{being}\ \mathrm{as} \ \mathrm{above} \Bigg\} . $$
\end{definition}

By arguing as in \cite{Ky} one can show that 
\begin{equation}\label{equivalence_atoms}
H^{\log}_{ \mathrm{at} } ( \mathbb{T} ) \cong H^{\log} (\mathbb{T}) .
\end{equation}

\subsection{Atomic decomposition of $H^{\log}_{\mathcal{I}} (\mathbb{T})$}

\begin{definition} Let $I $ be an arc in $ \mathcal{I}$.
A measurable function $a_I$ on $\mathbb{T}$ is said to be an $H^{\log}_{\mathcal{I}} $-atom associated to $I \in \mathcal{I}$ whenever
\begin{itemize}
\item $\mathrm{supp} (a_I) \subset I$,
\item $\int_I a_I ( \theta ) d \theta = 0$, and
\item $ \| a_I \|_{L^{\infty} (\mathbb{T})} \leq \| \chi_I \|_{L^{\log} (\mathbb{T})} $.
\end{itemize}
\end{definition}

In analogy with the non-dyadic case, define $H^{\log}_{\mathrm{at}, \mathcal{I}} (\mathbb{T})$ to be the class of all $f \in \mathcal{D}'$ for which there exists  a sequence $\{ \beta_{I_k} \}_{k \in \mathbb{N}}$ of scalar multiples of $H^{\log}_{\mathcal{I}} $-atoms  such that
$$ f - \widehat{f} (0) = \sum_{k \in \mathbb{N}} \beta_{I_k} \quad \text{in } \mathcal{D}' $$
and
$$ \sum_{k \in \mathbb{N}} | I_k | \Psi_0 \big( \| \beta_{I_k} \|_{L^{\infty} (\mathbb{T})} \big) < \infty .$$
For $f \in H^{\log}_{\text{at}, \mathcal{I}} (\mathbb{T})$, we set
$$ \| f \|_{H^{\log}_{\text{at}, \mathcal{I}} (\mathbb{T})} : = \inf \Bigg\{ \Lambda_{\infty} \big(f, \{ \beta_{I_k} \}_{k \in \mathbb{N}} \big) : f - \widehat{f} (0) =  \sum_{k \in \mathbb{N}} \beta_{I_k}  \text{ in } \mathcal{D}'  \Bigg\}, $$
where
$$ \Lambda_{\infty} \big( f, \{ \beta_{I_k} \}_{k \in \mathbb{N}} \big) : = \inf \Bigg\{ \lambda > 0 : \Psi_0 ( \lambda^{-1} |\widehat{f} (0)| ) + \sum_{k \in \mathbb{N}} | I_k | \Psi_0 \big( \lambda^{-1} \| \beta_{I_k} \|_{L^{\infty} (\mathbb{T})} \big) \leq 1 \Bigg\}. $$

We shall prove that $H^{\log}_{\mathcal{I}} (\mathbb{T})$ is contained in $H^{\log}_{\text{at}, \mathcal{I}} (\mathbb{T})$. To this end, we shall show first that every function $f \in F_{\mathcal{I}} (\mathbb{T})$ admits a decomposition in terms of $H^{\log}_{\mathcal{I}} $-atoms.

\begin{proposition}\label{at_dec} For every $f \in F_{\mathcal{I}} (\mathbb{T})$ there exists a finite collection of multiples of $H^{\log}_{ \mathcal{I}} $-atoms $\{ \beta_{I_k} \}_{k = 1}^N$  such that:
\begin{itemize}
\item $f - \widehat{f}(0) = \sum_{k=1}^N \beta_{I_k}$  and
\item $ \Lambda_{\infty} (f,  \{ \beta_{I_k} \}_{k=1}^N) \leq C_0 \| f \|_{H^{\log}_{\mathcal{I}} (\mathbb{T})} $,
\end{itemize}
where $C_0 > 0$ is an absolute constant that is independent of $f$. 
\end{proposition}

\begin{proof} The proof is a variant of the non-dyadic case presented in Chapter 1 of \cite{YLK}.

Fix an $f \in F_{\mathcal{I}} (\mathbb{T})$. Without loss of generality, we may assume that $\widehat{f} (0) = 0$. Notice that, since $\Psi_0$ satisfies the conditions of \cite[Section 3]{BG}, one has
\begin{equation}\label{m-s}
c_1 \int_{\mathbb{T}} \Psi_0 ( M_{\mathcal{I}} [f] ( \theta ) ) d \theta \leq \int_{\mathbb{T}} \Psi_0 ( S_{\mathcal{I}} [f] ( \theta ) ) d \theta \leq c_2 \int_{\mathbb{T}} \Psi_0 ( M_{\mathcal{I}} [f] ( \theta ) ) d \theta
\end{equation}
where $c_1, c_2 >0$ are absolute constants. Here, $M_{\mathcal{I}} [f]$ is the dyadic maximal function of $f$, namely
$$ M_{\mathcal{I}} [f] ( \theta ) := \sup_{N \in \mathbb{N}_0}|\mathbb{E}_{\mathcal{I},N} [f] ( \theta )  |, $$
where
$$ \mathbb{E}_{\mathcal{I},N} [f] : = \sum_{\substack{ I \in \mathcal{I}: \\ |I| = 2^{-N}} } \langle f \rangle_I \chi_I  .$$
Notice that, as $f \in F_{\mathcal{I}}(\mathbb{T})$, there exists an $N_0 \in \mathbb{N}$ such that  
\begin{equation}\label{fact}
\mathbb{E}_{\mathcal{I},N} [f] \equiv f \quad \text{ for all } N \geq N_0. 
\end{equation}
For $\lambda >0$, by arguing as on pp. 33--34 in \cite{Duo}, one can show that the set
$$ \Omega_{\lambda} := \{ \theta \in \mathbb{T} : M_{\mathcal{I}} [f] ( \theta ) > \lambda \}  $$
 can be written as a finite union of mutually disjoint arcs in $\mathcal{I}$. We may thus write $\Omega_{\lambda} = \bigcup_k I (\lambda, k)$, where the union is  finite and the arcs $\{ I (\lambda, k) \}_k$ are mutually disjoint and in $\mathcal{I}$.  We then define
$$ 
g_{\lambda} (\theta) := 
\begin{cases}
f(\theta) \quad &\text{if } x \in \mathbb{T} \setminus \Omega_{\lambda}, \\
\langle f \rangle_I \quad &\text{if } \theta \in I (\lambda, k) 
\end{cases}
$$
and $b_{\lambda}: =  f - g_{\lambda}  = \sum_{k} b_{\lambda, k} $, where $b_{\lambda, k} := \chi_{I (\lambda, k)} b $. By using \eqref{fact}, one can easily check that
\begin{equation}\label{fact_bound}
\| g_{\lambda} \|_{L^{\infty} (\mathbb{T})} \leq \lambda.
\end{equation}

Note that for $N \in \mathbb{N}$ large enough, one has $\Omega_{2^n} = \emptyset$ for all $n \geq N$ and so, $g_{2^n} \equiv f$ for all $n \geq N$. We thus have
$$ f = \sum_{n=0}^{N-1} (g_{2^{n+1}} - g_{2^n} ) =  \sum_{n=0}^{N-1} ( b_{2^n} - b_{2^{n+1}} ).   $$
We write $ \Omega_{2^n} = \bigcup_k I (2^n, k)$ and define $ \beta_{n,k} :=  ( b_{2^n} - b_{2^{n+1}} ) \chi_{ I (2^n, k) }  $. It can easily be seen that $  \int_{I (2^n,k)} \beta_{n,k} (x') d x' = 0$. Moreover, it follows from \eqref{fact_bound} that
$$ \| \beta_{n,k} \|_{ L^{\infty} (\mathbb{T})} \leq \| g_{2^n} \|_{L^{\infty} (\mathbb{T})} + \| g_{2^{n+1}} \|_{L^{\infty} (\mathbb{T})} \leq 3 \cdot 2^n.  $$
Hence, $\beta_{n,k}$ are multiples of $H^{\log}_{\mathcal{I}}$-atoms and moreover,
\begin{align*}
\sum_{n=0}^{N-1} \sum_k | I (2^n, k) | \Psi_0 (  \| \beta_{n,k} \|_{L^{\infty} (\mathbb{T})} / \lambda ) & \lesssim \sum_{n=0}^{N-1} \sum_k | I (2^n, k) | \Psi_0 (  2^n / \lambda ) \\
& =  \sum_{n=0}^{N-1} \Psi_0 (  2^n / \lambda )  \sum_k | I (2^n, k) |  \\
& =  \sum_{n=0}^{N-1} \Psi_0 (  2^n / \lambda ) \sum_{l \in \mathbb{N}_0} |\{ 2^{n+l} < M_{\mathcal{I}} [f] \leq 2^{n+l+1} \} | .  
\end{align*}
Note that there exists a $c_0 > 0$ such that $ \Psi_0 (2^{-l} t) \leq 2^{-c_0 l} \Psi_0 (t) $ for all $t \geq 0$ and $l \in \mathbb{N}_0$. Hence, 
\begin{align*}
 & \sum_{n=0}^{N-1} \Psi_0 ( 2^n / \lambda ) \sum_{l \in \mathbb{N}_0} |\{ 2^{n+l} < M_{\mathcal{I}} [f]   \leq 2^{n+l+1} \} |   \\=
   &  \sum_{l \in \mathbb{N}_0} \sum_{n=0}^{N-1} \Psi_0 ( 2^n / \lambda )  |\{ 2^{n+l} < M_{\mathcal{I}} [f]   \leq 2^{n+l+1} \} |  \\ \leq
 &  \sum_{l \in \mathbb{N}_0} \sum_{m \geq l}  \Psi_0 ( 2^{m-l} / \lambda )  |\{ 2^{m} < M_{\mathcal{I}} [f]   \leq 2^{m+1} \} |  \\ \leq
 &  \sum_{l \in \mathbb{N}_0} 2^{-c_0 l }\sum_{m \geq l}  \Psi_0 ( 2^m / \lambda )  |\{ 2^{m} < M_{\mathcal{I}} [f]   \leq 2^{m+1} \} |  \lesssim  \int_{\mathbb{T}} \Psi_0 ( \lambda^{-1}  M_{\mathcal{I}} [f]  (\theta)) d \theta ,
\end{align*}
which combined with \eqref{m-s}, completes the proof of the proposition.
\end{proof}

By arguing as on p. 109 in \cite{Big_Stein}, the density of $F_{\mathcal{I}} (\mathbb{T})$ in $H^{\log}_{\mathcal{I}} (\mathbb{T})$ and Proposition \ref{at_dec} imply the following result.

\begin{proposition}\label{at_dec_full_1}
One has $H^{\log}_{\mathcal{I}} (\mathbb{T}) \subseteq H^{\log}_{\mathrm{at}, \mathcal{I}} (\mathbb{T})$. 
\end{proposition}


\subsection{Proof of the reverse inclusion}

The main result of this section is that the converse of Proposition \ref{at_dec_full_1} also holds.

\begin{proposition}\label{at_dec_full_2}
One has $H^{\log}_{\mathrm{at}, \mathcal{I}} (\mathbb{T}) \subseteq H^{\log}_{  \mathcal{I}} (\mathbb{T})$. 
\end{proposition}

In order to prove Proposition \ref{at_dec_full_2}, we shall establish first the following dyadic variant of \cite[Lemma 1.3.5]{YLK}.

\begin{lemma}\label{atom_est}
Let  $I \in \mathcal{I}$ be a given arc. For every $L^{\infty}$-function $\beta_I$ supported in $I$ one has
 $$ \int_I \Psi_0  ( S_{\mathcal{I}} [\beta_I] ( \theta )  ) d \theta \leq \big( 1+ (2 \pi)^{1/2} \big) |I| \Psi_0 ( \| \beta_I \|_{L^{\infty} (\mathbb{T})}). $$
\end{lemma}

\begin{proof} We argue as in the proof of \cite[Lemma 1.3.5]{YLK}. More specifically,  since $\Psi$ is increasing and $\Psi_0 (c t) \leq  c \Psi_0 (t)$ for all $c \geq 1$ and $t \geq 0$, we have
\begin{align*}
\int_I \Psi_0 ( S_{\mathcal{I}} [\beta_I] ( \theta ) ) d \theta &= \int_I \Psi_0 \Big( \frac{ S_{\mathcal{I}} [\beta_I] (\theta)}{ \| \beta_I \|_{L^{\infty} (\mathbb{T})} } \| \beta_I \|_{L^{\infty} (\mathbb{T})}  \Big) d \theta \\
& \leq \int_I \Psi_0  \Big( \Big( 1 + \frac{ S_{\mathcal{I}} [\beta_I] (\theta)}{ \| \beta_I \|_{L^{\infty} (\mathbb{T})} }   \Big) \| \beta_I \|_{L^{\infty} (\mathbb{T})}  \Big) d \theta \\
& \leq \Psi_0 ( \| \beta_I \|_{L^{\infty} (\mathbb{T})} ) \int_I \Big(1 + \frac{ S_{\mathcal{I}} [\beta_I] (\theta)}{ \| \beta_I \|_{L^{\infty} (\mathbb{T})} } \Big)  d \theta \\
& =  \Psi_0 ( \| \beta_I \|_{L^{\infty} (\mathbb{T})} ) \Big( |I| + \| \beta_I \|_{L^{\infty} (\mathbb{T})}^{-1}  \int_I    S_{\mathcal{I}} [\beta_I] (\theta) d \theta \Big). 
\end{align*}
To complete the proof of the lemma, observe that by using the Cauchy-Schwarz inequality and the fact that $S_{\mathcal{I}}$ is an isometry on $L^2 (\mathbb{T})$, one has  
\begin{align*}
 \int_I  S_{\mathcal{I}} [\beta_I] ( \theta ) d \theta  \leq (2 \pi |I|)^{1/2} \| S_{\mathcal{I}} [\beta_I] \|_{L^2 (I)} & = (2 \pi |I|)^{1/2} \|  \beta_I \|_{L^2 (I)}  \\
 & \leq (2\pi)^{1/2} |I|  \|  \beta_I \|_{L^{\infty} (\mathbb{T})}, 
\end{align*}
as desired.
\end{proof}

\subsection{Proof of Proposition \ref{at_dec_full_2}}

Let $f $ be a given distribution in $H^{\log}_{\text{at}, \mathcal{I}} (\mathbb{T})$. Without loss of generality, we may assume that $\widehat{f} (0) = 0$. 

By definition, there exists a sequence of multiples of $H^{\log}_{\text{at}, \mathcal{I}}$-atoms $\{ \beta_{I_k } \}_{k \in \mathbb{N}}$ such that
$$ f = \lim_{N \rightarrow \infty} \sum_{k=1}^N \beta_{I_k} \text{  in }  \mathcal{D}' \quad \text{and} \quad \sum_{k=1}^{\infty} |I_k| \Psi_0 ( \| \beta_{I_k} \|_{L^{\infty} (\mathbb{T})} ) < \infty. $$
For $N \in \mathbb{N}$, we set $b_N := \sum_{k=1}^N \beta_{I_k}$. Note that since $b_N \in L^{\infty} (\mathbb{T})$ one has $b_N  = \sum_{I \in \mathcal{I}} \langle b_N, h_I \rangle h_I$ a.e. on $\mathbb{T}$ and in $L^2 (\mathbb{T})$.

In what follows, we shall use several times the fact that
\begin{equation}\label{sublinear}
\Psi_0 \Bigg( \sum_{j=1}^L t_j \Bigg) \leq \sum_{j=1}^L \Psi_0 (t_j)
\end{equation}
for any finite collection of non-negative numbers $\{ t_j \}_{=1}^L$; see the proof of \cite[Lemma 1.1.6 (i)]{YLK}. 

\begin{lemma}\label{proj_conv}
Let $I \in \mathcal{I}$ be given. If $\{ b_N \}_{N \in \mathbb{N}}$ is as above, then the sequence of complex numbers $\{ \langle b_N, h_I \rangle \}_{N \in \mathbb{N}}$ converges.
\end{lemma}

\begin{proof} It follows from Lemma \ref{atom_est} and \eqref{sublinear} that for $N>M$ one has
\begin{align*}
  \int_{\mathbb{T}} \Psi_0 (S_{\mathcal{I}} [b_N -b_M] (\theta) ) d \theta & \leq \sum_{k = M+1}^N \int_{\mathbb{T}} \Psi_0 (S_{\mathcal{I}} [ \beta_{I_k}] (\theta) ) d \theta   \\
  & \leq \big( 1+ (2 \pi)^{1/2} \big) \sum_{k = M+1}^N |I_k| \Psi_0 (\| \beta_{I_k} \|_{L^{\infty} (\mathbb{T})} ) 
\end{align*}
and so,
\begin{equation}\label{lim_Cauchy}
 \lim_{M,N \rightarrow \infty}  \int_{\mathbb{T}} \Psi_0 (S_{\mathcal{I}} [b_N - b_M] (\theta)) d \theta = 0 . 
\end{equation}
Observe that, since $\Psi_0$ is increasing, one has
\begin{align*}
\int_{\mathbb{T}} \Psi_0 (S_{\mathcal{I}} [b_N -b_M] (\theta) ) d \theta &\geq \int_{\mathbb{T}} \Psi_0 ( | \langle b_N , h_I \rangle - \langle b_M , h_I \rangle|  |I|^{-1/2} \chi_I ) d \theta \\
& = | I | \Psi_0 ( | \langle b_N , h_I \rangle -  \langle b_M , h_I \rangle|  |I|^{-1/2} )
\end{align*}
for all $I \in \mathcal{I}$. Since $\Psi_0$ is continuous and $\Psi_0 (t) = 0$ if, and only if, $t = 0$, we deduce from the previous inequality and \eqref{lim_Cauchy} that
$$  \lim_{M,N \rightarrow \infty} | \langle b_N , h_I \rangle - \langle b_M , h_I \rangle| = 0 .$$
Hence, $\{ b_N \}_{N \in \mathbb{N}}$ is Cauchy in $\mathbb{C}$ and so, it converges. 
\end{proof}

In view of Lemma \ref{proj_conv}, we may define $\mathbf{b} : = \{ b_I \}_{I \in \mathcal{I}}$ with $b_I := \lim_{N \rightarrow \infty} \langle b_N, h_I \rangle$. 
We claim that $\mathbf{b} \in h^{\log}_{\mathcal{I}} (\mathbb{T})$ with 
\begin{equation}\label{S_b}
\int_{\mathbb{T}} \Psi_0 ( S_{\mathcal{I}} [\mathbf{b}] (\theta) ) d \theta \leq \big( 1+ (2 \pi)^{1/2} \big) \sum_{k=1}^{\infty} | I_k | \Psi_0 ( \| \beta_{I_k} \|_{L^{\infty} (\mathbb{T})}) .
\end{equation}
Indeed, by Lemma \ref{atom_est} and the definition of $\{ b_N \}_{N \in \mathbb{N}}$, one has
\begin{equation}\label{S_b_N_uniform}
 \int_{\mathbb{T}} \Psi_0 ( S_{\mathcal{I}} [ b_N] (\theta) ) d\theta  \leq \big( 1+ (2 \pi)^{1/2} \big) \sum_{k = 1}^{\infty} |I_k| \Psi_0 (\|  \beta_{I_k} \|_{L^{\infty} (\mathbb{T})} ) \quad \text{for all } N \in \mathbb{N}.
\end{equation}
Fix an $M \in \mathbb{N}$ and note that, by combining \eqref{S_b_N_uniform} with Fatou's lemma, one gets
\begin{align*}
& \int_{\mathbb{T}} \liminf_{N \rightarrow \infty} \Psi_0 \Bigg( \Bigg\{ \sum_{\substack{I \in \mathcal{I}:\\ |I| \geq 2^{-M}}} | \langle b_N , h_I \rangle |^2 |I|^{-1} \chi_I(\theta) \Bigg\}^{1/2} \Bigg) d \theta \\
& \leq \liminf_{N \rightarrow \infty} \int_{\mathbb{T}}  \Psi_0 ( S_{\mathcal{I}} [ b_N] (\theta) ) d \theta \\
&  \leq \big(  1 + ( 2 \pi )^{1/2} \big) \sum_{k = 1}^{\infty} |I_k| \Psi_0 (\|  \beta_{I_k} \|_{L^{\infty} (\mathbb{T})} ) . 
\end{align*}
Since $\Psi_0$ is continuous, we deduce that 
\begin{multline}\label{S_b_uniform}
\int_{\mathbb{T}}  \Psi_0 \Bigg( \Bigg\{ \sum_{\substack{I \in \mathcal{I}:\\ |I| \geq 2^{-M}}} |  b_I  |^2 |I|^{-1} \chi_I ( \theta) \Bigg\}^{1/2} \Bigg)  d \theta \leq \\
\big( 1 + ( 2 \pi )^{1/2} \big) \sum_{k = 1}^{\infty} |I_k| \Psi_0 (\|  \beta_{I_k} \|_{L^{\infty} (\mathbb{T})} )  
\end{multline}
for all $M \in \mathbb{N}$. Hence, \eqref{S_b} is obtained by using \eqref{S_b_uniform}, the monotone convergence theorem, and the continuity of $\Psi_0$.

If we now define the sequence of functions $\{ B_M \}_{M \in \mathbb{N}}$ with finite wavelet expansions given by
$$ B_M (\theta) := \sum_{\substack{ I \in \mathcal{I} : \\  |I| \geq 2^{-M}} } b_I h_I (\theta) \quad (\theta \in \mathbb{T}), $$
then, as explained in Section \ref{prelim}, $B_M$ converges in $\mathcal{D}'$ to some $b \in \mathcal{D}'$.  
It thus suffices to prove that $b \equiv f $. To this end, it is enough to show that, in view of the definition of $\{ b_N \}_{N \in \mathbb{N}}$, one has
\begin{equation}\label{final_limit}
 b = \lim_{N \rightarrow \infty} b_N \text{ in } \mathcal{D}' . 
\end{equation}
For $ M , N \in \mathbb{N}$, consider the function $\delta_{M,N}$ given by
$$ \delta_{M,N} (\theta): = B_M (\theta) - (b_N )_M (\theta) \quad ( \theta \in \mathbb{T}), $$
where $(b_N )_M$ is the `truncation' of the Haar series representation of $b_N$ allowing Haar projections corresponding to arcs $ I \in \mathcal{I} $ with $  |I| \geq 2^{-M}$ that is,
$$ (b_N )_M (\theta) : = \sum_{\substack{ I \in \mathcal{I} : \\  |I| \geq 2^{-M}} } \langle b_N , h_I \rangle h_I (\theta)  \quad ( \theta \in \mathbb{T} ) . $$
We claim that for any fixed $\phi \in C^{\infty} (\mathbb{T})$ one has
\begin{equation}\label{lim_d_1}
 \lim_{M \rightarrow \infty}  \langle \delta_{M,N} , \phi \rangle =  \langle b- b_N, \phi \rangle  \quad \text{uniformly in } N \in \mathbb{N} .  
\end{equation}
Indeed, fix a $\phi \in C^{\infty} (\mathbb{T})$ and write
\begin{equation}\label{triangle}
| \langle \delta_{M,N} - (b- b_N), \phi \rangle | \leq | \langle B_M - b , \phi \rangle | + | \langle (b_N)_M - b_N, \phi \rangle | . 
\end{equation}
Let $p \in (1/2,1)$ be fixed. By arguing as in Section \ref{prelim}, one deduces that there exists an absolute constant $c_0 >0$ such that
\begin{equation}\label{triangle_1}
 | \langle B_M - b , \phi \rangle | \leq c_0 \| \phi' \|_{L^{\infty} (\mathbb{T})} D(p, \mathbf{b}) \sum_{k= M+1}^{\infty} 2^{-(2-1/p)k},
\end{equation}
where $D(p, \mathbf{b}) $ is as in Section \ref{prelim} and is finite, in view of \eqref{S_b}. Similarly, one has
\begin{equation}\label{triangle_2}
 | \langle (b_N)_M - b_N , \phi \rangle | \leq c_0 \| \phi' \|_{L^{\infty} (\mathbb{T})} D(p, \mathbf{b}_N) \sum_{k= M+1}^N 2^{-(2-1/p)k},
\end{equation}
where
$$ D(p, \mathbf{b}_N) = (p-1)^{-1/p} \Bigg( \int_{\mathbb{T}} \Psi_0 ( S_{\mathcal{I}} [\mathbf{b}_N] (\theta) ) d\theta \Bigg)^{1/p} . $$
By using Lemma \ref{atom_est} and the definition of $b_N$, one deduces that
\begin{align*}
 D(p, \mathbf{b}_N) & \leq (p-1)^{-1/p} \Bigg(  \big( 1 + ( 2 \pi )^{1/2}  \big) \sum_{k=1}^{N} | I_k | \Psi_0 ( \| \beta_{I_k} \|_{L^{\infty} (\mathbb{T})} ) \Bigg)^{1/p} \\
 & \leq \big( 1 + ( 2 \pi )^{1/2}\big)^{1/p} (p-1)^{-1/p} \Bigg(  \sum_{k=1}^{\infty} | I_k | \Psi_0 ( \| \beta_{I_k} \|_{L^{\infty} (\mathbb{T})} ) \Bigg)^{1/p}  
\end{align*}
that is, $D(p, \mathbf{b}_N) $ is bounded by a finite constant that is independent of $N \in \mathbb{N}$. Hence, \eqref{triangle_2} implies that
\begin{multline}\label{triangle_2_b}
| \langle (b_N)_M - b_N , \phi \rangle | \leq \\
c_p \| \phi' \|_{L^{\infty} (\mathbb{T})}   \Bigg(  \sum_{k=1}^{\infty} | I_k | \Psi_0 ( \| \beta_{I_k} \|_{L^{\infty} (\mathbb{T})} ) \Bigg)^{1/p} \sum_{k= M+1}^{\infty} 2^{-(2-1/p)k},
\end{multline}
where $c_p>0$ is a constant depending only on $p$. Therefore, by combining \eqref{triangle_1} with \eqref{triangle_2_b}, we deduce that \eqref{lim_d_1} holds. 

Moreover, by arguing again as in Section \ref{prelim}, one shows that for any fixed $p \in (1/2, 1)$ there exists a constant $c'_p > 0$, depending only on $p$,  such that
\begin{multline}\label{d_unif}
 |\langle \delta_{M, N} , \phi \rangle| \leq \\
  c'_p \| \phi' \|_{L^{\infty} (\mathbb{T})} \Bigg[ \int_{\mathbb{T}}   \Psi_0  \Bigg( \Bigg\{ \sum_{\substack{I \in \mathcal{I}:\\ |I| \geq 2^{-M}}} |  b_I  - \langle b_N, h_I \rangle |^2 |I|^{-1} \chi_I ( \theta ) \Bigg\}^{1/2} \Bigg) d \theta \Bigg]^{1/p}   
\end{multline}
for all $M, N \in \mathbb{N} $ and $ \phi \in C^{\infty} (\mathbb{T})$. We shall prove that the right-hand side of \eqref{d_unif} tends to $0 $ as $N \rightarrow \infty$ for all $M \in \mathbb{N}$ and $ \phi \in C^{\infty} (\mathbb{T})$. To this end, note that for any $L \in \mathbb{N}$ and for every collection $\{ t_l \}_{l=1}^L$ of non-negative numbers, one has
\begin{equation}\label{upper_type}
\Psi_0 \Bigg[ \Big( \sum_{j=1}^L t_j \Big)^{1/2}\Bigg] \leq   \sum_{j=1}^L \Psi_0 ( t_j^{1/2} ). 
\end{equation}
Indeed, \eqref{upper_type} is obtained by combining \eqref{sublinear} with
$$ \Big( \sum_{j=1}^L t_j \Big)^{1/2}  \leq \sum_{ j =1}^L t^{1/2}_j . $$
Observe that by using \eqref{upper_type} one has
\begin{align*} 
& \int_{\mathbb{T}} \Psi_0 \Bigg( \Bigg\{ \sum_{\substack{I \in \mathcal{I}:\\ |I| \geq 2^{-M}}} |  b_I  - \langle b_N, h_I \rangle |^2 |I|^{-1} \chi_I (\theta) \Bigg\}^{1/2} \Bigg) d \theta \leq \\
& \sum_{\substack{I \in \mathcal{I}:\\ |I| \geq 2^{-M}}} \int_{\mathbb{T}}   \Psi_0  \big(   |  b_I  - \langle b_N, h_I \rangle |^2 |I|^{-1} \chi_I (\theta) \big) d \theta = \sum_{\substack{I \in \mathcal{I}:\\ |I| \geq 2^{-M}}} |I| \Psi_0 \big( |  b_I  - \langle b_N, h_I \rangle |  |I|^{-1/2} \big)  
\end{align*}
and hence, \eqref{d_unif} implies that
\begin{equation}\label{d_bound}
|\langle \delta_{M, N} , \phi \rangle| \leq  
 c'_p \| \phi' \|_{L^{\infty} (\mathbb{T})} \Bigg( \sum_{\substack{I \in \mathcal{I}:\\ |I| \geq 2^{-M}}} |I| \Psi_0 \big( | b_I - \langle b_N, h_I \rangle | |I|^{-1/2} \big) \Bigg)^{1/p} .
\end{equation}
Since the sum on the right-hand side of \eqref{d_bound} is finite and $\Phi$ is continuous, it follows from the definition of $\mathbf{b}$ that
\begin{equation}\label{lim_d_2}
 \lim_{N \rightarrow \infty} \langle \delta_{M,N} , \phi \rangle = 0 \quad \text{for all } M \in \mathbb{N} .  
\end{equation}
Therefore, by combining \eqref{lim_d_1} and \eqref{lim_d_2}, it follows that 
$$ \lim_{N \rightarrow \infty} \langle b - b_N, \phi \rangle = \lim_{N \rightarrow \infty} \lim_{M \rightarrow \infty} \langle \delta_{M,N} , \phi \rangle = \lim_{M \rightarrow \infty} \lim_{N \rightarrow \infty} \langle \delta_{M,N} , \phi \rangle = 0 $$
for all $\phi \in C^{\infty} (\mathbb{T})$. Hence, \eqref{final_limit} holds and so, the proof of Proposition \ref{at_dec_full_2} is complete.

\subsection{Concluding remarks}
By combining Propositions \ref{at_dec_full_1} and \ref{at_dec_full_2} one obtains the following theorem.

\begin{theo}\label{at_dec_full}
One has $H^{\log}_{ \mathrm{at} , \mathcal{I} } (\mathbb{T}) \cong H^{\log}_{  \mathcal{I}} (\mathbb{T})$. 
\end{theo}

By using \cite[Proposition 2.1]{Mei} and the atomic decomposition of $H^{\log} (\mathbb{T})$; see  \eqref{equivalence_atoms}, one shows that 
 $H^{\log}_{\text{at}, \mathcal{I}^0} (\mathbb{T}) + H^{\log}_{\text{at}, \mathcal{I}^{1/3}} (\mathbb{T}) \cong H^{\log} (\mathbb{T})$. We thus deduce from  Theorem \ref{at_dec_full}  the following variant of T. Mei's theorem \cite{Mei} for $H^{\log} (\mathbb{T})$. 
 
\begin{theo}\label{two_translates}
One has
$$ H^{\log}_{ \mathcal{I}^0 } (\mathbb{T}) + H^{\log}_{\mathcal{I}^{1/3}} (\mathbb{T}) \cong H^{\log} (\mathbb{T}). $$
\end{theo}

\begin{rmk} Let $\mathcal{I}$ be a given system of dyadic arcs in $\mathbb{T}$. For $p \in (0, \infty)$, define the dyadic Hardy space $h^p_{\mathcal{I}} (\mathbb{T})$ as the class of all collections of complex numbers  
$ \mathbf{f}  = \{ f_I \}_{I \in \mathcal{I}} \cup \{ f_0\} $ such that $S_{\mathcal{I}} [ \mathbf{f} ] \in L^p (\mathbb{T})$. 

By arguing as in Sections \ref{intro_d} and \ref{d_atom}, one can show that $h^p_{\mathcal{I}} (\mathbb{T})$ can be identified with a dyadic $H^p$ space $H^p_{\mathcal{I}} (\mathbb{T})$ of distributions on $\mathbb{T}$ and moreover, the  following extension of T. Mei's theorem  \cite{Mei} holds
$$  H^p_{\mathcal{I}^0} (\mathbb{T}) + H^p_{\mathcal{I}^{1/3}} (\mathbb{T}) \cong H^p  (\mathbb{T}) \quad \text{for all } p \in (1/2, 1] . $$
We remark that dyadic Hardy spaces for $p<1$ have also been considered in \cite{W} and \cite{MP}, but the definitions there are different than ours. 
\end{rmk}


\section{Proof of the Bonami-Grellier-Ky theorem in the periodic setting}\label{d_BGK}

For a function $b \in L^1 (\mathbb{T})$, we set
$$ \| b \|_{BMO^+_{\mathcal{I}} (\mathbb{T})} : = \Big( \sup_{I \in \mathcal{I}} \frac{1}{|I|^2} \int_I | b(\theta) - \langle b \rangle_I|^2 d \theta \Big)^{1/2} + \Big| \int_{\mathbb{T}} b (\theta ) d \theta \Big| . $$
We then define $BMO^+_{\mathcal{I}} (\mathbb{T})$ as the class of all functions $b \in L^1 (\mathbb{T})$ such that $\| b \|_{BMO^+_{\mathcal{I}} (\mathbb{T})} < \infty  $. One defines  $BMO^+ (\mathbb{T})$ similarly.

Recall the following standard consequence of the John-Nirenberg type result in the dyadic case.
\begin{lemma}\label{J-N}
There exists an absolute constant $C_0 > 0$ such that for every function $b \in BMO^+_{\mathcal{I}} (\mathbb{T})$ one has
$$ \| b \|_{\exp L (\mathbb{T})} \leq C_0 \| b \|_{BMO^+_{\mathcal{I}} (\mathbb{T})}, $$
where $\| b \|_{\exp L (\mathbb{T})} : = \inf \{ \lambda > 0 : \int_{\mathbb{T}} \psi ( |b(x)| /\lambda ) dx \leq 1\}  $ and $\psi (t):= e^t -  t -1$, $t \geq 0$.
\end{lemma}

The following variant of \cite[Proposition 2.1]{BGK} is obtained by combining Lemma \ref{J-N} with \cite[Lemma 2.1]{BGK}.

\begin{proposition}\label{rough_bound}
For all functions $f$, $b$ such that $f \in L^1 (\mathbb{T}) $ and $b \in BMO^+_{\mathcal{I}} (\mathbb{T})$, one has
$$ \| f \cdot b \|_{L^{\log} (\mathbb{T})} \lesssim \| f \|_{L^1 (\mathbb{T})} \| b \|_{BMO^+_{\mathcal{I}} (\mathbb{T})} . $$
\end{proposition}

In this section, we present the following dyadic version of \cite[Theorem 1.1]{BGK}.

\begin{theo}\label{1-d_periodic} There exist two bilinear operators $S$ and $T$ on the product space $H^1_{\mathcal{I}} (\mathbb{T}) \times BMO^+_{\mathcal{I}} (\mathbb{T})$ such that 
$$ f \cdot b = S_{\mathcal{I}} (f,b) + T_{\mathcal{I}} (f,b) \quad \mathrm{in } \ \mathcal{D}' $$
with $S_{\mathcal{I}} : H^1_{\mathcal{I}} (\mathbb{T}) \times BMO_{\mathcal{I}}^+ (\mathbb{T}) \rightarrow L^1 (\mathbb{T})$ and $T_{\mathcal{I}} : H^1_{\mathcal{I}} (\mathbb{T}) \times BMO^+_{\mathcal{I}} (\mathbb{T}) \rightarrow H^{\log}_{\mathcal{I}} (\mathbb{T})$.
\end{theo}

The proof of Theorem \ref{1-d_periodic} that we present here is a variant of the corresponding one given by Bonami, Grellier, and Ky in \cite{BGK} (that establishes \cite[Theorem 1.1]{BGK}). To be more specific, let $f \in H^1_{\mathcal{I}} (\mathbb{T})$ be a function with finite wavelet expansion. If $b$ is a function in $\mathrm{BMO}^+_{\mathcal{I}} (\mathbb{T})$ that also has a finite wavelet expansion, then we may write
$$ f \cdot b = \Pi_1 (f,b) + \Pi_2 (f,b) + \Pi_3 (f,b), $$
where
$$ \Pi_1 (f,b) (\theta) := \sum_{\substack{I,J \in \mathcal{I} :\\ J \supsetneq I}} f_I g_J h_I (\theta) h_J (\theta) ,$$
$$ \Pi_2 (f,b) (\theta) := \sum_{\substack{I,J \in \mathcal{I} :\\ I \supsetneq J}} f_I g_J h_I (\theta) h_J (\theta) ,$$
and
$$ \Pi_3 (f,b) (\theta) :=  \sum_{\substack{I,J \in \mathcal{I} :\\  I = J }} f_I g_J  h_I (\theta) h_J (\theta). $$
We shall prove that:
\begin{itemize}
\item $ \Pi_1 $ can be extended as a bounded bilinear operator from $ H^1_{\mathcal{I}} (\mathbb{T}) \times BMO^+_{\mathcal{I}} (\mathbb{T}) $ to $ H^{\log}_{\mathcal{I}} (\mathbb{T})$, 
\item $ \Pi_2 $  can be extended as a bounded bilinear operator from $ H^1_{\mathcal{I}} (\mathbb{T}) \times BMO^+_{\mathcal{I}} (\mathbb{T}) $ to $ H^1_{\mathcal{I}} (\mathbb{T})$, and 
\item $ \Pi_3 $ can be extended as a bounded bilinear operator from $ H^1_{\mathcal{I}} (\mathbb{T}) \times BMO^+_{\mathcal{I}} (\mathbb{T}) $ to $ L^1 (\mathbb{T})$. 
\end{itemize}
One can thus conclude that Theorem \ref{1-d_periodic} holds by taking $S_{\mathcal{I}} := \Pi_1$ and $T_{\mathcal{I}}: = \Pi_2 + \Pi_3$. 


\begin{proposition}\label{P3}
The bilinear operator $\Pi_3$ extends into a bounded bilinear operator from $H^1_{\mathcal{I}} (\mathbb{T}) \times BMO^+_{\mathcal{I}} (\mathbb{T})$ to $L^1 (\mathbb{T})$.
\end{proposition}

\begin{proof}
It is well-known that $H^1_{\mathcal{I}} (\mathbb{T})$ admits a characterisation in terms of atoms. More specifically, recall that a measurable function $a $ is said to be an $H^1_{\mathcal{I}} (\mathbb{T})$-atom if it is either the constant function or there exists an $\Omega \in \mathcal{I}$ such that $\mathrm{supp} (a) \subseteq \Omega$, $\int_{\mathbb{T}} a_{\Omega} (\theta) d \theta = 0$, and $ \| a \|_{L^2 (\mathbb{T})} \leq |\Omega|^{-1/2} $. Then, $f \in H^1_{\mathcal{I}} (\mathbb{T})$ if, and only if, there exist a sequence $(\lambda_k)_k$ of non-negative scalars and a sequence $(a_k)_k$ of $H^1_{\mathcal{I}} (\mathbb{T})$-atoms such that
$$ f = \sum_{k \in \mathbb{N}} \lambda_k a_k $$
in the $H^1_{\mathcal{I}}$-norm and moreover, one has
$$ \| f \|_{H^1_{\mathcal{I}} (\mathbb{T})} \sim \inf \Bigg\{  \sum_{k \in \mathbb{N}}  | \lambda_k | : f = \sum_k \lambda_k a_k  \Bigg\}. $$

Hence, to show that $\Pi_3$ maps $H^1_{\mathcal{I}} (\mathbb{T}) \times BMO^+_{\mathcal{I}} (\mathbb{T})$ to $L^1 (\mathbb{T})$, it is enough to prove that there exists a constant $C>0$ such that
\begin{equation}\label{P3_atom}
\| \Pi_3 (a, b) \|_{L^1 (\mathbb{T})} \leq C \| b \|_{BMO^+_{\mathcal{I}} (\mathbb{T})}
\end{equation}
for all non-constant $H^1_{\mathcal{I}} (\mathbb{T})$-atoms $a$ and every $b \in BMO^+_{\mathcal{I}} (\mathbb{T})$ with $\int_{\mathbb{T}} b (\theta) d \theta = 0$ such that $a$ and $b$ have finite wavelet expansions. To this end, assume that $a = \sum_{I \subseteq \Omega} a_I h_I$ has finite wavelet expansion and is associated to some $\Omega \in \mathcal{I}$. Then,
$$  \Pi_3 (a, b) (\theta) = \sum_{\substack{I,J \in \mathcal{I} :\\  I = J }} a_I b_J  h_I (\theta) h_J (\theta) =  \sum_{I \subseteq \Omega} a_I b_I \frac{\chi_I (\theta)}{|I|}  $$
and hence, by using the Cauchy-Schwarz inequality, one gets the pointwise estimate
$$ | \Pi_3 (a, b) (\theta) | \leq S_{\mathcal{I}} [ a ] (\theta) \cdot S_{\mathcal{I}} [ P_{\Omega} b ] (\theta) ,$$
where $P_{\Omega} b (\theta) := \sum_{I \subseteq \Omega} b_I h_I (\theta) $.
We thus have by using the Cauchy-Schwarz inequality and the $L^2$-boundedness of $S_{\mathcal{I}}$,
\begin{equation}\label{L^2_Omega} 
\| \Pi_3 (a, b) \|_{L^1 (\mathbb{T})} \leq \| a\|_{L^2 (\mathbb{T})} \| P_{\Omega} b  \|_{L^2 (\mathbb{T})} . 
\end{equation}
Since $ \| a\|_{L^2 (\mathbb{T})} \leq |\Omega|^{-1/2}$ and $ \| P_{\Omega} b \|_{L^2 (\mathbb{T})} \leq |\Omega|^{1/2} \| b \|_{BMO^+_{\mathcal{I}} (\mathbb{T})}$,  \eqref{P3_atom} follows from \eqref{L^2_Omega}. \end{proof}


\begin{proposition}\label{P2}
The bilinear operator $\Pi_2$ extends into a bounded bilinear operator from $H^1_{\mathcal{I}} (\mathbb{T}) \times BMO^+_{\mathcal{I}} (\mathbb{T})$ to $H^1_{\mathcal{I}} (\mathbb{T})$.
\end{proposition}

\begin{proof} As in the proof of the previous proposition, it suffices to prove that there exists an absolute constant $C>0$ such that 
\begin{equation}\label{P2_atom}
\| \Pi_2 (a, b) \|_{H^1_{\mathcal{I}} (\mathbb{T})} \leq C \| b \|_{BMO^+_{\mathcal{I}} (\mathbb{T})}
\end{equation}
for each non-constant $H^1_{\mathcal{I}} (\mathbb{T})$-atom $a$ and for each $b \in BMO^+_{\mathcal{I}} (\mathbb{T})$ with $\int_{\mathbb{T}} b (\theta) d \theta  = 0$ such that $a$ and $b$ have finite wavelet expansions. Towards this aim, write $a = \sum_{I \subseteq \Omega} a_I h_I$ for some $\Omega \in {\mathcal{I}}$ and notice that
\begin{align*}
  \Pi_2 (a,b) (\theta) =  \sum_{\substack{I,J \in \mathcal{I} :\\  \Omega \supseteq I \supsetneq J}} a_I b_J h_I (\theta) h_J (\theta) & =  \sum_{\substack{I,J \in \mathcal{I} :\\  \Omega \supseteq I \supsetneq J}} a_I (P_{\Omega} b)_J h_I (\theta) h_J (\theta) \\
  & = \sum_{\substack{I,J \in \mathcal{I} :\\  \Omega \supseteq I \supsetneq J}} a_I (P_{\Omega} b)_J h_I (c_J) h_J (\theta) ,
\end{align*}
where $c_J$ denotes the centre of $J$. Moreover, observe that since
$$ \langle a \rangle_J = |J|^{-1} \int_J a =  \sum_{ \substack{ I \in \mathcal{I} :\\  I \subseteq \Omega }} a_I  |J|^{-1}  \int_J h_I (\theta) d \theta  =  \sum_{ \substack{ I \in \mathcal{I} :\\  J \subsetneq I \subseteq \Omega }} a_I h_I (c_J) , $$
one may rewrite $\Pi_2 (a,b) $ as
$$  \Pi_2 (a,b) (\theta) =  \sum_{J \in \mathcal{I}}  \langle a \rangle_J ( P_{\Omega} b )_J h_J (\theta). $$
Hence,
$$ S_{\mathcal{I}} [  \Pi_2 (a,b) ] (\theta) = \Big(  \sum_{J \in \mathcal{I}} | \langle a \rangle_J |^2 | ( P_{\Omega} b)_J  |^2 \frac{ \chi_J (\theta)}{|J|} \Big)^{1/2} \leq M (a) (\theta) \cdot S_{\mathcal{I}}  [ P_{\Omega} b  ] (\theta), $$
where $M$ denotes the Hardy-Littlewood maximal operator acting on functions defined over $\mathbb{T}$. Hence, by using the Cauchy-Schwarz inequality and the $L^2$-boundedness of $M$ and $S_{\mathcal{I}}$, one gets
$$ \| \Pi_2 (a, b) \|_{H^1_{\mathcal{I}} (\mathbb{T})}  = \|  S_{\mathcal{I}} \big( \Pi_2 (a,b) \big) \|_{L^1 (\mathbb{T})} \lesssim \| a \|_{L^2 (\mathbb{T})} \| P_{\Omega} b \|_{L^2 (\mathbb{T})} .$$
As in the proof of the previous lemma, note that one has $ \| a\|_{L^2 (\mathbb{T})} \leq |\Omega|^{-1/2}$ and $ \| P_{\Omega} b \|_{L^2 (\mathbb{T})} \leq |\Omega|^{1/2} \| b \|_{BMO^+_{\mathcal{I}} (\mathbb{T})}$ and so, \eqref{P2_atom} follows from the last estimate.
\end{proof}


It follows from Propositions \ref{P3} and \ref{P2} that if we define 
$$ T(f,b) (\theta) : =  \Pi_2 (f,b) ( \theta ) + \Pi_3 (f,b) ( \theta ) ,$$
 then $T$ is a bilinear operator that maps $H^1_{\mathcal{I}} (\mathbb{T}) \times BMO^+_{\mathcal{I}} (\mathbb{T})$ to $L^1 (\mathbb{T})$. Therefore, to complete the proof of Theorem \ref{1-d_periodic}, it remains to handle $\Pi_3$.


\begin{proposition}\label{P1}
The bilinear operator $\Pi_1$ extends into a bounded bilinear operator from $H^1_{\mathcal{I}} (\mathbb{T}) \times BMO^+_{\mathcal{I}} (\mathbb{T})$ to $H^{\log}_{\mathcal{I}} (\mathbb{T})$.
\end{proposition}

\begin{proof} Fix an $f \in H^1_{\mathcal{I}} (\mathbb{T}) $ and $b \in BMO^+_{\mathcal{I}} (\mathbb{T})$ with finite wavelet expansions and moreover, assume that $\int_{\mathbb{T}} b (\theta) d \theta = 0$.

First of all, arguing as above, one may write
$$ \Pi_1 (f,b) (\theta) =  \sum_{ I \in \mathcal{I} } f_I \langle b \rangle_I h_I (\theta) . $$
Let $a$ be a non-constant $H^1_{\mathcal{I}} (\mathbb{T})$-atom with finite wavelet expansion that is supported in some $\Omega \in \mathcal{I}$ so that $\| a \|_{L^2 (\mathbb{T})} \leq |\Omega|^{-1/2}$. We claim that
\begin{equation}\label{dec}
\Pi_1 (a,b) (\theta) = \Pi_1 (a, P_{\Omega} b) (\theta) +  \langle b \rangle_{\Omega} \cdot a (\theta).
\end{equation}
Indeed, to see this, write
$$ \Pi_1 (a,b) (\theta) =  \sum_{ I \in \mathcal{I} } a_I \langle b \rangle_I  h_I (\theta) =\Pi_1 (a, P_{\Omega} b) (\theta)  + \sum_{ I \in \mathcal{I} } a_I \langle b - P_{\Omega}b \rangle_I h_I (\theta) 
$$
and observe that
\begin{align*}
\sum_{ I \in \mathcal{I} } a_I \langle b - P_{\Omega}b \rangle_I h_I (\theta) &=  \sum_{ I \in \mathcal{I} } a_I \Big( |I|^{-1} \int_I \sum_{\substack{J \in \mathcal{I}:\\ J \supsetneq \Omega }} b_J h_J (\theta') d \theta'  \Big) h_I (\theta)  \\
&=  \sum_{ I \in \mathcal{I} } a_I \Big(   \sum_{\substack{I \in \mathcal{D}:\\ J \supsetneq \Omega }} b_J h_J (c_{\Omega}) \Big) h_I (\theta) \\
&= \Big( \sum_{\substack{J \in \mathcal{I}:\\ J \supsetneq \Omega }} b_J h_J (c_{\Omega}) \Big) a (\theta)  \\
&= \langle b \rangle_{\Omega} \cdot a (\theta),
\end{align*}
where $c_{\Omega}$ denotes the centre of $\Omega$. Hence, the proof of \eqref{dec} is complete. 

We may assume that $f = \sum_{k=1}^N \lambda_k a_k$, where each atom $a_k$ has a finite wavelet expansion. By using \eqref{dec}, one may write
$$ \Pi_1 (f, b) = \beta_1 + \beta_2 ,$$
where 
$$ \beta_1 (\theta) := \sum_{k=1}^N \lambda_k  \Pi_1 (a_k, P_{\Omega_k} b) (\theta) $$
and
$$ \beta_2 ( \theta ) :=  \sum_{k=1}^N \lambda_k \langle b \rangle_{\Omega_k} a_k ( \theta)  . $$
For the first term we have
\begin{align*}
\| \Pi_1 (a_k, P_{\Omega_k} b) \|_{H^1_{\mathcal{I}} (\mathbb{T})} &=  \| S_{\mathcal{I}} [ \Pi_1 (a_k, P_{\Omega_k} b)  ]  \|_{L^1 (\mathbb{T})} \\
&= \Big\| \Big( \sum_{\substack{I \in \mathcal{I} \\ I \subseteq \Omega_k} } | m_I (P_{\Omega}) |^2 |(a_k)_I|^2 \frac{\chi_I }{ | I | } \Big)^{1/2} \Big\|_{L^1 (\mathbb{T})}  \\
& \leq \| M (P_{\Omega_k} b) S_{\mathcal{I}} [ a_k ]  \|_{L^1 (\mathbb{T})} \\
& \leq \| M (P_{\Omega_k} b) \|_{L^2 (\mathbb{T})} \| S_{\mathcal{I}} [ a_k ] \|_{L^2 (\mathbb{T})} \\
& \leq C \| P_{\Omega_k} b \|_{L^2 (\mathbb{T})} \| a_k \|_{L^2 (\mathbb{T})} \\
& \leq C |\Omega_k|^{1/2} \|   b \|_{BMO^+_{\mathcal{I}} (\mathbb{T})} |\Omega_k|^{-1/2} = C \|   b \|_{BMO^+_{\mathcal{I}} (\mathbb{T})}
\end{align*}
for all $k=1,\cdots, N$. Hence,
$$  \| \beta_1 \|_{H^1_{\mathcal{I}} (\mathbb{T})} \leq \sum_{k=1}^N |\lambda_k| \| \Pi_1 (a_k, P_{\Omega_k} b) \|_{H^1_{\mathcal{I}} (\mathbb{T})} \leq C \| b \|_{BMO^+_{\mathcal{I}} (\mathbb{T})} \sum_{k=1}^N |\lambda_k| $$
and so, one deduces that
$$  \| \beta_1 \|_{H^1_{\mathcal{I}} (\mathbb{T})} \lesssim \| f \|_{H^1_{\mathcal{I}} (\mathbb{T})}  \| b \|_{BMO^+_{\mathcal{I}} (\mathbb{T})} .$$
It remains to treat $\beta_2$. The goal is to prove that
\begin{equation}\label{beta_1}
\| S_{\mathcal{I}} [ \beta_2 ]  \|_{L^{\log} (\mathbb{T})} \lesssim \| b \|_{BMO^+_{\mathcal{I}} (\mathbb{T})} \sum_{k=1}^N | \lambda_k | ,
\end{equation}
where the implied constant is independent of $b$, $f$ (and $N$). To this end, observe that 
\begin{align*}
S_{\mathcal{I}} [ \beta_2 ] ( \theta )  & \leq \sum_{k=1}^N |\lambda_k| | \langle b \rangle_{\Omega_k} | S_{\mathcal{I}} [ a_k ] (\theta) \\
& \leq \sum_{k=1}^N |\lambda_k| |b(\theta) -\langle b \rangle_{\Omega_k} | S_{\mathcal{I}} [ a_k ] (\theta) + |b(\theta)| \sum_{k=1}^N |\lambda_k|  S_{\mathcal{I}} [ a_k ] (\theta) \\
&= \sum_{k=1}^N |\lambda_k| |P_{\Omega_k} b (s)| S_{\mathcal{I}} [ a_k ]  (\theta) + |b(\theta)| \sum_{k=1}^N |\lambda_k|  S_{\mathcal{I}} [ a_k ] (\theta) .
\end{align*}
Hence,
\begin{align*}
 \| S_{\mathcal{I}} [ \beta_2 ]  \|_{L^{\log} (\mathbb{T})} & \lesssim \Big\| \sum_{k=1}^N |\lambda_k| |P_{\Omega_k} b  | S_{\mathcal{I}} [ a_k ]  \Big\|_{L^{\log} (\mathbb{T})} +  \Big\|  |b| \sum_{k=1}^N |\lambda_k|  S_{\mathcal{I}} [ a_k ]  \Big\|_{L^{\log} (\mathbb{T})}   \\
 & \lesssim \Big\| \sum_{k=1}^N |\lambda_k| |P_{\Omega_k} b  | S_{\mathcal{I}} [ a_k ]  \Big\|_{L^1 (\mathbb{T})} +  \Big\|  |b| \sum_{k=1}^N |\lambda_k|  S_{\mathcal{I}} [ a_k ]  \Big\|_{L^{\log} (\mathbb{T})} .
\end{align*}
By arguing as above, it can easily be seen that
$$ \Big\| \sum_{k=1}^N |\lambda_k| |P_{\Omega_k} b | S_{\mathcal{I}} [ a_k ]  \Big\|_{L^1 (\mathbb{T})}  \lesssim  \| b \|_{BMO^+_{\mathcal{I}} (\mathbb{T})} \sum_{k=1}^N |\lambda_k|. $$
Therefore, the proof of \eqref{beta_1} is reduced to showing that
\begin{equation}\label{bmo_rough}
 \Big\|  |b| \sum_{k=1}^N |\lambda_k|  S_{\mathcal{I}} [ a_k ]  \Big\|_{L^{\log} (\mathbb{T})}  \lesssim \| b \|_{BMO^+_{\mathcal{I}} (\mathbb{T})} \sum_{k=1}^N |\lambda_k| .
\end{equation}
To this end, note that by using Proposition \ref{rough_bound} one gets
$$ \Big\|  |b| \sum_{k=1}^N |\lambda_k|  S_{\mathcal{I}} [ a_k ]  \Big\|_{L^{\log} (\mathbb{T})}  \lesssim \| b \|_{BMO^+_{\mathcal{I}} (\mathbb{T})} \Big\| \sum_{k=1}^N |\lambda_k|  S_{\mathcal{I}} [ a_k ]  \Big\|_{L^1 (\mathbb{T})}  $$
and since for each $H^1_{\mathcal{I}}$-atom one has $\| S_{\mathcal{I}} [ a_k] \|_{L^1 (\mathbb{T})} \leq 1$, \eqref{bmo_rough} follows from the last estimate.
This completes the proof of \eqref{beta_1}.
\end{proof}

\subsection{Passing from dyadic to non-dyadic decompositions}\label{translate} Assume that $f \in H^1 (\mathbb{T})$ and $b \in BMO^+ (\mathbb{T})$. Then, as shown in \cite{Mei} one has 
\begin{equation}\label{BMO_translates}
 BMO (\mathbb{T}) = BMO_{\mathcal{I}^0} (\mathbb{T}) \cap  BMO_{\mathcal{I}^{1/3}} (\mathbb{T}) 
\end{equation}
and there exist  $f_1 \in H^1_{\mathcal{I}^0} (\mathbb{R})$ and $f_2 \in H^1_{\mathcal{I}^{1/3}} (\mathbb{R})$ such that $f = f_1 + f_2$.  
Having fixed such a decomposition of $f$, write
$$ f \cdot b = f_1 \cdot b + f_2 \cdot g = S_{\mathcal{I}^0} (f_1, b) + T_{\mathcal{I}^0} (f_1, b) + S_{\mathcal{I}^{1/3}} (f_2, b) + T_{\mathcal{I}^{1/3}} (f_2, b) \quad \text{in } \mathcal{D}'. $$
Hence, a periodic version of Theorem \ref{B-G-K_Thm} is obtained by taking  
$$ S(f, b) : = S_{\mathcal{I}^0} (f_1, b) + S_{\mathcal{I}^{1/3}} (f_2, b)  $$
and
$$ T(f, b) : = T_{\mathcal{I}^0} (f_1, b) + T_{\mathcal{I}^{1/3}} (f_2, b)  . $$


\section{Some further remarks in the periodic setting} \label{extensions}

\subsection{Variants of Theorems \ref{Stein-type_lemma} and \ref{LlogL} on $\mathbb{T}$}\label{periodic_2-3}
There is a periodic version of Theorem \ref{Stein-type_lemma}, namely $M(f) \in L_{\Psi_0} (\mathbb{T})$ if, and only if, $f \in L \log \log L (\mathbb{T})$. Combining this with Lemma \ref{decreasing}, one obtains the following result.

\begin{proposition}\label{inclusion_LloglogL} If 
$f \in  L \log \log L (\mathbb{T}) $, then $ f \in  H^{\log} (\mathbb{T}) $.
\end{proposition}

Moreover, arguing as in the Section \ref{zygmundproof} and using the necessity in (an appropriate periodic version of) Theorem \ref{Stein-type_lemma} as well as Proposition \ref{inclusion_LloglogL} and Lemma \ref{decreasing}, one can show that if $f $ is a non-negative function in $H^{\log} (\mathbb{T})$, then $f \in L \log \log L (\mathbb{T})$. 

\begin{proposition}\label{identification} 
One has $$ \{ f \in L \log \log L (\mathbb{T}) : f\geq 0 \  \mathrm{a.e.}\ \mathrm{on}\ \mathbb{T}\} = \{ f\in  H^{\log} (\mathbb{T}) : f \geq 0\  \mathrm{a.e.}\ \mathrm{on}\ \mathbb{T} \} .$$
\end{proposition}

\begin{proof} Note that Proposition \ref{inclusion_LloglogL} implies that
\begin{equation}\label{incl_1}
\{ f \in L \log \log L (\mathbb{T}) : f\geq 0 \  \mathrm{a.e.}\ \mathrm{on}\ \mathbb{T}\} \subseteq \{ f \in  H^{\log} (\mathbb{T})   : f \geq 0\  \mathrm{a.e.}\ \mathrm{on}\ \mathbb{T} \} .
\end{equation}

To prove the reverse inclusion, take a non-negative function $f$ in $ H^{\log} (\mathbb{T})$ and notice that it follows from the work of Stein \cite{Stein_LlogL} that
\begin{equation}\label{wt_reverse_periodic} 
 |\{ \theta \in \mathbb{T} : M (f) ( \theta ) > c_1\alpha \} | \geq\frac{c_2}{\alpha} \int_{\{ |f| > \alpha \}} |f( \theta )| d \theta, 
\end{equation}
where $c_1,c_2 > 0$ are absolute constants. Hence, by arguing as in the proof of Theorem \ref{Stein-type_lemma}, it follows from \eqref{wt_reverse_periodic} (noting that the periodic case is easier as one does not need to consider the contribution away from the support of $f$) that
\begin{equation}\label{reverse_periodic}
 \int_{\mathbb{T}} {\Psi}_0 (M (f)) (\theta) d \theta \gtrsim 1 + \int_{\mathbb{T}} |f(x)| \log^+ \log^+ |f(\theta)| d \theta . 
\end{equation}
Since $f \geq 0$ a.e. on $\mathbb{T}$, as in the Euclidean case, one has 
\begin{equation}\label{pw_per}
f^{\ast} (\theta) \geq  \sup_{0< r < 1} | ( P_r \ast f  ) (\theta) | \gtrsim M (f) (\theta) \quad \mathrm{for}\  \mathrm{a.e.}\ \theta \in \mathbb{T},
\end{equation}
where $P_r$ denotes the Poisson kernel in the periodic setting. Hence, by using \eqref{reverse_periodic}, \eqref{pw_per}, and Lemma \ref{decreasing}, we deduce that $f \in L \log \log L (\mathbb{T})$ and so,
\begin{equation}\label{incl_2}
\{ f \in  H^{\log} (\mathbb{T})  : f \geq 0\  \mathrm{a.e.}\ \mathrm{on}\ \mathbb{T} \} \subseteq
\{ f \in L \log \log L (\mathbb{T}) : f\geq 0 \  \mathrm{a.e.}\ \mathrm{on}\ \mathbb{T}\}.
\end{equation}
The desired fact is a consequence of \eqref{incl_1} and \eqref{incl_2}.
\end{proof}

\subsubsection{{Some further applications}}\label{applications_periodic}

We conclude with some applications of Theorem \ref{generalstein} in the periodic setting.
The function 
$$ \Psi(x,t)=\Psi  (t)=t \log^+t\,\log^+\log^+t $$
appearing in \cite{Sjolin} satisfies the hypotheses of Theorem \ref{generalstein}, and we now determine which space maps into $L_{\Psi}$ via the maximal function. With the associated $\psi$ defined as before, an integration by parts yields
$$ \int \frac{\psi(s)}{s}ds=\frac{1}{2}(\log^+s)^2\log^+\log^+s+\log^+s\log^+\log^+s-\frac{1}{4}(\log^+s)^2. $$
This allows us to conclude that, for this choice of $\Psi$, 
$$ M(f)\in L_{\Psi}(\mathbb{T})\quad \textrm {if, and only if,}\quad  f\in L\log^2L\log\log L(\mathbb{T}). $$
Turning to the space $L\log\log L \log\log\log\log L$ appearing in Lie's paper \cite{Lie}, we can check where the maximal operator maps this space. Performing the appropriate computations, we obtain that
$$ \int_{\mathbb{T}}\frac{M(f)  }{\log(M(f)  +e)}\log^+\log^+\log^+\log^+M(f)  d \theta<\infty $$
if, and only if,
$$ f\in L\log \log L\log \log \log \log L (\T). $$

Roughly speaking, the contents of Theorem \ref{generalstein} and the computations presented above can be summarised as follows. Let $\Phi_0$ be a given Orlicz function, namely $\Phi_0 : [0 , \infty) \rightarrow [0, \infty) $ is an increasing function with $\Phi_0 (0) = 0$ and $\Phi_0 (t) \rightarrow \infty$ as $t \rightarrow \infty$. Suppose that one can find non-negative, increasing functions $M,S$ with
$$ \Phi_0 (t) = M (t) \cdot S(t) \quad (t>0)$$
and such that, for $0 < \alpha < t$, one can easily compute
$$ F_{\alpha} (t) := \int_{\alpha}^t \frac{M'(s)}{s} d s $$
in closed form and, moreover, that there exists an $\alpha_0 >0$ with the property that for every $\alpha \geq \alpha_0$ one has
$$ F_{\alpha} (t) \cdot S(t) \gtrsim  \int_{\alpha}^t \Big( \frac{M (s)}{s} + F (s) \Big) \cdot S' (s) d s  \quad \mathrm{for}\ \mathrm{all} \ t \geq \alpha . $$
Then, by arguing as in Section \ref{steinproof}, one deduces the `concrete' relation
$$ f \in L_{\Phi_0} (\mathbb{T}) \quad \mathrm{if}, \ \mathrm{and}\ \mathrm{only}\ \mathrm{if}, \quad M(f) \in L_{F_{\alpha} \cdot S} (\mathbb{T}),$$
for any $\alpha \geq \alpha_0$.


\subsection{A variant of an inequality of Hardy and Littlewood for $H^{\log} (\mathbb{T})$}\label{FC_H^{log}}

 A classical result due to Hardy and Littlewood asserts that for every $p \in (0,1]$ there exists a constant $C_p >0$ such that  
\begin{equation}\label{p < 1_disc}
\Bigg( \sum_{n = 1 }^{\infty} \frac{| f_n |^p}{ |n|^{2-p}} \Bigg)^{1/p} \leq  C_p \norm{ F }_{H^p (\mathbb{D})}
\end{equation}
for all analytic functions $F (z) = \sum_{n=0}^{\infty} f_n z^n$ in the Hardy space $H^p (\mathbb{D})$ on the unit disc $\mathbb{D}$; \cite[Theorem 16]{H-L}. It follows from  \eqref{p < 1_disc} that for every $p \in (0,1]$ there exists a constant $B_p >0$
\begin{equation}\label{p < 1}
\Bigg( \sum_{n\in\mathbb{Z} \setminus \{0\}} \frac{ | \widehat{f} (n)|^p}{ |n|^{2-p}} \Bigg)^{1/p} \leq B_p \norm{ f }_{H^p (\mathbb{T})} .
\end{equation}

Since $H^{\log}(\mathbb{T}) \subset H^p (\mathbb{T})$, $p \in (0,1)$ (see Remark \ref{p-LlogL}), one deduces from \eqref{p < 1} that $\{ |n|^{p-2} \widehat{f} (n)|^p \}_{n \in \mathbb{Z} \setminus \{0\}}$ is summable for any $f \in H^{\log}(\mathbb{T})$ for all $p \in (0,1)$.  The next theorem establishes a more accurate description of the behaviour of the Fourier coefficients of distributions in $H^{\log} (\mathbb{T})$.

\begin{theo}\label{H-L_H^{log}}
There exists a constant $C>0$ such that
$$ \sum_{n \in \mathbb{Z} \setminus \{ 0 \} } \frac{ \Psi_0 (|n \widehat{f} (n)| )}{ n ^2 } \leq C \int_{\mathbb{T}} \Psi_0 ( f^{\ast} (\theta) ) d \theta,  $$
where $\Psi_0 (t) : = t \cdot [\log (e+t)]^{-1}$, $t \geq 0$. 
\end{theo}

\begin{proof} We shall prove that there exists an absolute constant $C_0> 0 $ such that
\begin{equation}\label{atom_ineq}
\sum_{ n \in \mathbb{Z} \setminus \{ 0 \} } \frac{ \Psi_0 (|n \widehat{a_I} (n)| )}{ n ^2 } \leq C_0 |I| \Psi_0 ( \| a_I \|_{L^{\infty} (\mathbb{T})} )
\end{equation}
for any $L^{\infty}$-function $a_I$ supported in some arc $I$ in $\mathbb{T}$ with $\int_I a_I (\theta) d \theta = 0$.

To this end, we fix such a function $a_I$ (and an arc $I$) and write
$$ \sum_{ n \in \mathbb{Z} \setminus \{ 0 \} } \frac{ \Psi_0 (|n \widehat{a_I} (n)| )}{ n ^2 } =  A + B , $$
where
$$ A : = \sum_{1 \leq |n| \leq |I|^{-1} } \frac{ \Psi_0 (|n \widehat{a_I} (n)| )}{ n ^2 } $$
and
$$ B: = \sum_{ |n| > |I|^{-1} } \frac{ \Psi_0 (|n \widehat{a_I} (n)| )}{ n ^2 } . $$
We shall prove that
\begin{equation}\label{A_bound}
A \lesssim |I|  \Psi_0 \big( \| a_I \|_{L^{\infty}(\mathbb{T})} \big)
\end{equation}
and
\begin{equation}\label{B_bound}
B \lesssim |I|  \Psi_0 \big( \| a_I \|_{L^{\infty}(\mathbb{T})} \big).
\end{equation}
To prove \eqref{A_bound}, by using the cancellation of $a_I$ and the fact that $ | e^{- i n x } - e^{-i n y } | \leq |n| | x-y | \leq |n| |I|$ for all $n \in \mathbb{Z}$ and $x,y \in I$, one has
\begin{equation}\label{a_I-bound}
|\widehat{a_I}(n)| \leq  |n| |I|^2 \|a_I \|_{L^{\infty} (\mathbb{T})} \quad \text{for all } n \in \mathbb{Z}. 
\end{equation}
Since $\Psi_0$ is increasing and there exists an absolute constant $A_0 > 0$ such that $\Psi_0 (st) \leq A_0 s^{2/3} \Psi_0 (t)$ for all $t>0$ and $s \in (0,1)$; see \cite[Example 1.1.5 (i)]{YLK}, it follows from \eqref{a_I-bound} that
\begin{align*}
A  \leq \sum_{1 \leq |n| \leq |I|^{-1}} \frac{ \Psi_0 \big(n^2 |I|^2 \| a_I \|_{L^{\infty}(\mathbb{T})} \big) }{n^2}
&\lesssim  \sum_{1 \leq n \leq |I|^{-1}} (|I|^2n^2)^{2/3} \frac{ \Psi_0 \big( \| a_I \|_{L^{\infty}(\mathbb{T})} \big) }{n^2} \\
& = |I|^{4/3} \Psi_0 \big( \| a_I \|_{L^{\infty}(\mathbb{T})} \big) \sum_{1 \leq n \leq |I|^{-1}} n^{-2/3} \\
& \lesssim |I| \Psi_0 \big( \| a_I \|_{L^{\infty}(\mathbb{T})} \big) .
\end{align*}
Hence, \eqref{A_bound} holds. To establish \eqref{B_bound}, note that by using H\"older's inequality for $p=4$ and $p'=4/3$ and Parseval's identity, one obtains
$$ B \leq B_1 \cdot B_2, $$
where 
$$ B_1 : = \| a_I \|_{L^2 (\mathbb{T})}^{1/2} \leq |I|^{1/4} \| a_I \|_{L^{\infty} (\mathbb{T})}^{1/2} $$
and
$$ B_2 := \Bigg( \sum_{ |n| > |I|^{-1} } \frac{\widetilde{\Psi}_0 \big( |n \widehat{a_I} (n)| \big) } {n^2} \Bigg)^{3/4} $$
with $\widetilde{\Psi}_0 (t) : = t^{2/3} [\log(e+t)]^{-4/3}$, $t \geq 0$. Since $ |\widehat{a_I} (n)| \leq |I| \| a_I \|_{L^{\infty} (\mathbb{T})} $ and $\widetilde{\Psi}_0$ is increasing on $[0, \infty)$, we have
\begin{align*}
B_2 & \leq \Bigg( \sum_{ |n| > |I|^{-1} } \frac{\widetilde{\Psi}_0 \big( |n| |I| \| a_I \|_{L^{\infty} (\mathbb{T})} \big) } {n^2} \Bigg)^{3/4} \\
    & \leq |I|^{1/2} \frac{ \| a_I \|_{L^{\infty} (\mathbb{T})}^{1/2} } { \log \big( e + \| a_I \|_{L^{\infty} (\mathbb{T})} \big) }  \Bigg( \sum_{ |n| > |I|^{-1} }  n^{-4/3} \Bigg)^{3/4} \\
    & \lesssim |I|^{3/4} \frac{ \| a_I \|_{L^{\infty} (\mathbb{T})}^{1/2} } { \log \big( e + \| a_I \|_{L^{\infty} (\mathbb{T})} \big) }
\end{align*}
and so, \eqref{B_bound} holds as
$$ B \leq B_1 \cdot B_2 \lesssim |I|  \Psi_0 \big( \| a_I \|_{L^{\infty}(\mathbb{T})} \big) . $$
Therefore, in view of \eqref{A_bound} and \eqref{B_bound}, \eqref{atom_ineq} holds. 

To complete the proof of the theorem, take an $f \in H^{\log } (\mathbb{T})$ and note that there exists a sequence $\{ b_{I_k} \}_{k \in \mathbb{N}}$ of multiples of atoms in $H^{\log} (\mathbb{T})$, supported in arcs $I_k$, such that
$$  f - \widehat{f}(0) = \sum_{k \in \mathbb{N}} b_{I_k} \quad \text{in } \mathcal{D}'$$
and
$$ \sum_{k \in \mathbb{N}} |I_k | \Psi_0 \big( \| b_{I_k} \|_{L^{\infty} (\mathbb{T})} \big) \leq A \int_{\mathbb{T}} \Psi_0 ( f^{\ast} (\theta) ) d \theta , $$
where $A>0$ is an absolute constant. Hence, by using \eqref{sublinear} and \eqref{atom_ineq} we get 
\begin{align*}
\sum_{n \in \mathbb{Z} \setminus \{ 0 \} } \frac{ \Psi_0 ( |n \widehat{f} (n)| )}{ n ^2 } & \leq
\sum_{k \in \mathbb{N}} \sum_{n \in \mathbb{Z} \setminus \{ 0 \} } \frac{ \Psi_0 ( |n \widehat{b_{I_k}} (n) | )}{ n ^2 } \\
& \lesssim \sum_{k \in \mathbb{N}} |I_k | \Psi_0 \big( \| b_{I_k} \|_{L^{\infty} (\mathbb{T})} \big) \\
& \lesssim \int_{\mathbb{T}} \Psi_0 ( f^{\ast} (\theta) ) d \theta 
\end{align*}
and this completes the proof of our theorem.
\end{proof}

\begin{rmk}
One deduces from Theorem \ref{H-L_H^{log}} that for any $f\in H^{\log} (\mathbb{T})$,
\begin{equation}\label{convergence}
    \sum_{n\in \mathbb{Z}\setminus \{0\}} \frac{|\widehat{f}(n)|}{|n|\log(e+|n|)}<\infty. 
\end{equation}
Indeed, observe that 
\begin{equation}\label{dual}
| \widehat{f}(n)| \lesssim_f 1 + |n| \quad \text{ for all } n \in \mathbb{Z}.
\end{equation}
To see this, note that $( H^p (\T))^{\ast} \cong \Lambda_{ p^{-1}-1} (\T)$ for $p <1$; see \S 7.4 in \cite{Duren} and so, for $f \in H^{\log} (\T) \subset H^{2/3} (\T)$ one has
\begin{align*}
\abs{ \langle f , e_n \rangle } \lesssim_f \norm{ e_n }_{ \Lambda_{1/2 } (\T)}  = \norm{e_n}_{L^{\infty}(\T)} + \sup_{ \substack{  x,y \in [0, 2\pi): \\ x\neq y}}   \frac{ \abs{ e^{i  n x} - e^{i  n y}}} {\abs{x-y}^{1/2}} 
& \lesssim 1 + \abs{n}^{1/2} \\
& \le 1 + \abs{n}
\end{align*}
for all $n \in \mathbb{Z}$, where $e_n (x) : = e^{i n x}$, $x \in \mathbb{T}$.  
Therefore, in view of Theorem \ref{H-L_H^{log}} and \eqref{dual}, \eqref{convergence} holds. 
\end{rmk}
Theorem \ref{H-L_H^{log}} can be used to exhibit distributions in $ H^p (\mathbb{T}) \setminus H^{\log} (\mathbb{T})$ for $p\in (0,1)$. For instance, it follows from Theorem \ref{H-L_H^{log}} that the Dirac distribution $\delta_0$ does not belong to $ H^{\log} (\mathbb{T})$.

Furthermore, Theorem \ref{H-L_H^{log}} is sharp in the following sense: if $\widetilde{\Psi} : [ 0 , \infty) \rightarrow [0, \infty)$ is any increasing function with $ \lim_{t \rightarrow \infty} \widetilde{\Psi} (t) / \Psi_0 (t) = \infty$, then there is no constant $C>0$ such that
\begin{equation}\label{false}
\sum_{n \in \mathbb{Z} \setminus \{ 0 \} } \frac{ \widetilde{\Psi} (|n \widehat{f} (n)| )}{ n ^2 } \leq C \int_{\mathbb{T}} \Psi_0 ( f^{\ast} (\theta) ) d \theta
\end{equation}
for all $f \in H^{\log} (\mathbb{T})$. Indeed, take a function $\widetilde{\Psi}$ as above and suppose that \eqref{false} holds true. Let $N$ be a large positive integer that will eventually be sent to infinity. Consider the function 
$$
 a_N (\theta) : = N 2^N  e^{i 2^N \theta}  \chi_{[0, 2\pi 2^{-N})} (\theta), \quad \theta \in [ 0 , 2\pi). 
$$
One can easily check that 
\begin{equation}\label{a_N_atom}
\| a_N \|_{H^{\log} (\mathbb{T})} \lesssim 1,
\end{equation}
where the implied constant is independent of $N$. 

Consider the interval $I_N : = [2^{N-2}, 2^{N-1})$ and observe that there exists an absolute constant $c_0 >0$ such that for every natural number $n$ in $ I_N $ one has
$$
|\widehat{a_N} (n) |  = N 2^{N} \frac{ \Big| e^{-i 2\pi (2^{-N}n - 1 )} -1 \Big|}{2\pi |n-2^N|} 
 =  N  2^N   \frac{\Big|\sin \big[ \pi \big(n 2^{-N} -1 \big) \big]\Big| }{2\pi ( 2^N - n)}  \geq c_0 N,
$$
where we used the identity $|e^{is} - e^{it} | = |\sin [(s-t)/2]|$ for $s=-2\pi (2^{-N} n -1)$ and $t=0$ as well as the fact that for $n \in I_N$ one has $2^N- n \sim 2^N$ and $| \sin [ \pi (2^{-N} n -1 ) ] | \sim 1$. 

Hence, \eqref{false} and \eqref{a_N_atom} imply that
\begin{align*}
1 \gtrsim \sum_{n \in \mathbb{Z} \setminus \{ 0 \} } \frac{ \widetilde{\Psi} (|n \widehat{a_N} (n)| )}{ n ^2 }  \ge  \sum_{n = 2^{N-2} }^{2^{N-1}} \frac{ \widetilde{\Psi} (|n \widehat{a_N} (n)| )}{ n ^2 } 
& \ge \widetilde{\Psi} (c_0 2^{N-2} N ) \sum_{n = 2^{N-2} }^{2^{N-1}} n^{-2} \\
& \sim \frac{ \widetilde{\Psi} (c_0 2^{N-2} N ) } {2^N} \\
& = \frac{ \Psi_0 (c_0 2^{N-2} N) }{2^N} \cdot \frac{ \widetilde{\Psi} (c_0 2^{N-2} N )}{\Psi_0 (c_0 2^{N-2} N)} \\
& \sim  \frac{ \widetilde{\Psi} (c_0 2^{N-2} N )}{\Psi_0 (c_0 2^{N-2} N)},
\end{align*} 
which yields a contradiction by taking $N \in \mathbb{N}$ `large enough'.

\begin{rmk} 
As a consequence of sharpness of Theorem \ref{H-L_H^{log}} discussed above, one deduces that
$$ \sum_{n \in \mathbb{Z} \setminus \{ 0 \} } \frac{ |\widehat{f} (n) |}{ |n| \big[ \log \big( e + |n \widehat{f}(n) | \big) \big]^s} \lesssim \int_{\mathbb{T}} \Psi_0 ( f^{\ast} (\theta) ) d \theta $$
is false when $0<s<1$.
\end{rmk}

It follows from Theorem \ref{H-L_H^{log}} and \cite[(8.3)]{BIJZ} that there exists a constant $C>0$ such that
\begin{equation}\label{H^log_disc}
  \sum_{n = 1 }^{\infty} \frac{ \Psi_0 \big( n | f_n | \big)}{ n^2 }  \leq  C \sup_{0<r<1} \int_0^{2\pi} \Psi_0 \big( |F(r e^{i\theta})| \big) d \theta
\end{equation}
for all analytic functions $F (z) = \sum_{n=0}^{\infty} f_n z^n$ in the unit disc $\mathbb{D}$ for which the quantity on the right-hand side of \eqref{H^log_disc} is finite.

We remark that variants of \eqref{p < 1_disc} and \eqref{p < 1} for certain classes of Hardy-Orlicz spaces have been obtained in \cite{P} and \cite{WW} (see also \cite{Jain, R, VK}), which do not include the case of $H^{\log}(\mathbb{T})$ treated above. Moreover, our methods are completely different from those in the aforementioned references.


\subsection*{Acknowledgments}
AS extends his thanks to Kelly Bickel and the rest of the Department of Mathematics at Bucknell University (Lewisburg, PA) for hospitality during a visit where part of this work was carried out.


 

\end{document}